\documentclass[11pt, dvipsnames]{amsart}

\usepackage{kantlipsum} 

\calclayout

\usepackage{graphics,graphicx}

\newtheorem{theorem}{Theorem}

\newtheorem{example}{Example}%
\newtheorem{remark}{Remark}%
\newtheorem{lemma}{Lemma}

\numberwithin{equation}{section}
\usepackage{float}
\usepackage{booktabs}
\usepackage{multirow}
\usepackage{color}
\usepackage{xcolor}
\usepackage{amsmath}
\usepackage{amssymb}

\newcommand{\R}{\mathbb{R}} 
 \newcommand{\bm}[1]{\boldsymbol{#1}}
\newcommand{\dv}{\mathrm{div}}
\newtheorem{setting}{Setting}[section]
\newtheorem{assumption}{Assumption}[section]
\usepackage[foot]{amsaddr}

\begin{document}

\title[PINNs Error Analysis ]{A Unified Framework for the Error Analysis of Physics-Informed Neural Networks}

\begin{abstract}
We prove a priori and a posteriori error estimates for physics-informed neural networks (PINNs) for linear PDEs. We analyze elliptic equations in primal and mixed form, elasticity, parabolic, hyperbolic and Stokes equations; and a PDE constrained optimization problem. For the analysis, we propose an abstract framework in the common language of bilinear forms, and we show that coercivity and continuity lead to error estimates. The obtained estimates are sharp and reveal that the $L^2$ penalty approach for initial and boundary conditions in the PINN formulation weakens the norm of the error decay. Finally, utilizing recent advances in PINN optimization, we present numerical examples that illustrate the ability of the method to achieve accurate solutions.

   \vspace{1em}
    \smallskip 
\end{abstract}
\author{Marius Zeinhofer$^1$, Rami Masri$^1$, and  Kent--Andr\'{e} Mardal$^2$}

\email{mariusz@simula.no, rami@simula.no, kent-and@math.uio.no}
\address{$^1$Department of Numerical Analysis and Scientific Computing, Simula Research Laboratory, Oslo, 0164 Norway. MZ and RM acknowledge support from the Research Council of Norway (grant number 303362) and via FRIPRO (grant number 324239). 
}
\address{$^2$Department of Mathematics, University of Oslo, Norway.}
\date{ \today }

\maketitle

\section{Introduction} 
Employing neural networks as ansatz functions for the solution of PDEs goes back to the works of \cite{dissanayake1994neural, lagaris1998artificial} where residual minimization of the strong form of the PDE -- now typically referred to as the PINN formulation -- was proposed. Embracing increased computational power due to the usage of GPUs, this approach was later popularized and further developed by \cite{raissi2019physics, karniadakis2021physics, jagtap2020adaptive}. Other formulations, based on variational principles or stochastic formulations have been proposed in \cite{yu2018deep, han2018solving, han2017deep}. 

Neural networks are increasingly investigated for use in computational fluid dynamics \cite{cai2021physics, jin2021nsfnets, cai2021artificial}, solid mechanics \cite{haghighat2021physics} and high-dimensional PDEs \cite{han2017deep, pfau2020ab, hermann2020deep}, to name but a few areas of ongoing research. In these application, meshfreeness and ease of implementation for both forward, inverse and parametric problems are the main motivation to use neural network based PDE solvers. With Nvidia Modulus, an industrial and highly optimized software framework for the implementation of physics-informed neural networks is freely available \cite{hennigh2021nvidia}.

\paragraph{Existing Literature}  
In the works \cite{mishraforward, de2022error, 10.1093/imanum/drac085}, the authors provide estimates, relating the loss function's value after training to the achieved error and required network size. More precisely,  
they bound the $L^2$ error by the loss function's value multiplied by a constant depending on the network. 
The treated problems include semilinear parabolic, Euler, Navier-Stokes and Kolmogorov equations. The paper \cite{shin2020error} focuses on linear elliptic PDEs and provides a priori estimates under the assumption of finding the global minimum in the neural network ansatz space. In \cite{muller2022notes}, the authors provide a priori and a posteriori estimates for linear elliptic and parabolic equations under the assumption of exactly satisfied boundary and initial conditions. Another line of work \cite{xu2020finite, muller2022error} focuses on the analysis of variational formulations of elliptic PDEs.

\paragraph{Main Contributions}
In this paper, we provide a priori and a posteriori error estimates of the PINN formulation of linear elliptic (Poisson, Darcy and elasticity), parabolic and hyperbolic equations. Furthermore, we analyze Stokes equations in steady state, and we consider a PDE constrained optimization problem for elliptic equations with distributed control. We extend the existing literature in the following ways:
\begin{itemize}

    \item We provide sharp error estimates in Sobolev norms. The sharpness of the error estimates reveals problems with the PINN formulation. Precisely, the $L^2$ penalty approach for boundary, initial and possible divergence constraints deteriorates the norm in which the error decays. For example, in the case of Poisson's equation, the sharpness of the following coercivity estimate
    \begin{equation}
\|u\|_{H^s(\Omega)}^2 \lesssim \|\Delta u \|^2_{L^2(\Omega)} + \|u\|_{L^2(\partial \Omega)}^2  \,\,\, \mathrm{iff} \,\,\,  s \leq 1/2,  \nonumber   
    \end{equation}
  dictates that the error decays at most in $H^{1/2}(\Omega)$. Similar issues  arise for all other equations.  
    Conversely, if feasible, encoding these constraints directly into the ansatz \cite{sukumar2022exact, richter2022neural} avoids these problems.

    \item Using recent advances in optimization algorithms \cite{muller2023achieving}, we present numerical simulations in three or four dimensions of the analyzed equations. In all cases, we are able to efficiently achieve highly accurate solutions with comparatively small network size without laborious hyperparameter tuning.
    
    \item In comparison to \cite{mishraforward, de2022error, 10.1093/imanum/drac085}, we relax the assumption on the solution's regularity. In the cited articles, the authors work with classical solutions and derive error estimates in the $L^2$ norm with constants depending on the neural network approximation, see for example \cite[Theorem 3.4]{mishraforward} where the constant $C_1$  depends on the neural network ansatz. The constants in our estimates do not depend on the neural network approximation.
   
    \end{itemize}

\paragraph{Proof Strategy}
The main technical difficulty in the analysis lies in the $L^2$ penalty approach for boundary, initial and possible divergence constraints. Inspecting energy and regularity estimates, it is clear that the $L^2$ norm is not the natural one occurring in these results. However, the $L^2$ penalty approach is prevailing in practice for its simplicity in implementations. This mismatch can be handled similarly for all considered equations in the following way. 1) For the PDE in question, we derive energy estimates with inhomogeneous right-hand side, boundary terms and possibly initial and divergence terms. 2) Assuming smooth enough data, we use regularity theory to derive dual estimates by integration by parts. 3) Interpolating between energy estimates and dual estimates, we obtain the coercivity properties required by our abstract framework which in turn allows to deduce a priori and a posteriori estimates.
 
The remainder of the article is structured as follows: In Section \ref{sec:prelims}, we introduce notation and required mathematical preliminaries. Section \ref{sec:abstract_framework} provides an introduction to physics-informed neural networks and presents the abstract framework used for error analysis. In Section \ref{sec:applications},  we prove the crucial coercivity and continuity estimates which show error estimates for all the aforementioned equations. Finally, in Section \ref{sec:numerics} we present numerical results.

\section{Preliminaries}\label{sec:prelims} Throughout the paper, we let $\Omega \subset \mathbb{R}^d$ be a bounded domain. We use standard notation for the Sobolev spaces $W^{m,p}(\Omega)$ for $m 
\in \mathbb N_0$ and $p \in [1,\infty]$. For $p=2$,  the Hilbert spaces are denoted by  $H^{m}(\Omega) = W^{m,2}(\Omega)$. For $s \in (0,1)$ and $m \in \mathbb{N}_0$, we recall the fractional spaces $H^{m+s}(\Omega)$ normed by 
\begin{equation} \label{eq:gagliardo_norm}
\|u\|_{H^{m+s}(\Omega)}^2 = \|u\|_{H^m(\Omega)}^2 + \sum_{|\alpha| = m} \int_{\Omega} \int_{\Omega} \frac{|D^{\alpha} u(x) - D^{\alpha} v(y)|^2}{|x - y|^{d + 2s}} \mathrm{d}x \mathrm{d}y,  
\end{equation}
where $\alpha$ is a mutli-index and $D^{\alpha}$ is the corresponding distributional derivative. For a Banach space $W$, we denote its dual by $W^*$.  In the context of Hilbert spaces, for a function $u\in L^2(\Omega)$, we mean for $k 
 \geq0$
\begin{equation*}
    \|u\|_{H^k(\Omega)^*} = \sup_{\|\varphi\|_{H^k(\Omega)}\leq 1} \int_\Omega u\varphi \,\mathrm dx, 
\end{equation*}
and similarly for $H^s_0(\Omega)$. This means that interpreting a function as an element of a dual space always uses the $L^2$ duality product. 
We also use fractional Bochner spaces, as defined in \cite[Section 2.5.d]{hytonen2016analysis}. Given a Banach space $X$, an interval $I \subset \R$, and  $u \in H^s(I, X)$, define 
\begin{equation}
\|u\|_{H^s(I,X)}^2  = \|u\|^2_{L^2(I,X)} + \int_I \int_I \frac{ \|u(t) - u(\tau )\|_X^2}{|t -\tau |^{1+2s}},  \quad s \in (0,1).  
\end{equation}
For $m \in \mathbb N_0$, $\|u\|_{H^m(I,X)}$ denotes the usual definition of Bochner norms. The space $H^m_0(I,X)$ denotes the space of functions vanishing at initial and final times.

We will often use results from real interpolation theory. Here, we present the  results that are necessary for our exposition, and we refer to \cite[Chapter 14]{brenner2008mathematical} and the references therein for more details. Given two Banach spaces $B_0, B_1$ with $B_1 \subset B_0$, the interpolated space $[B_0, B_1]_{\theta,p}$ with $0  < \theta < 1 $ and $p \in [1,\infty)$ is 
\begin{equation}
[B_0, B_1]_{\theta,p} = \left \{ u \in B_0 \mid \,\,  \int_0^\infty t^{-\theta p - 1  } K(t,u)^p \mathrm{d}t  < \infty \right \}, 
\end{equation}
where 
\[ K(t,u) = \inf_{v \in B_1}  ( \|u-v\|_{B_0} + t \|v\|_{B_1} ).  \]
For our analysis, the following result on linear operators and interpolated spaces is essential. Given Banach spaces $A_1 \subset A_0$ and $B_1 \subset B_0$. Let $T: A_0 \rightarrow B_0$  be a  linear continuous operator which is also continuous as a map from $A_1 \rightarrow B_1$, then there holds that $T$ is continuous from $[A_0, A_1]_{\theta,p} \rightarrow [B_0, B_1]_{\theta,p}$. Moreover,
\begin{equation} \label{eq:operator_interpolation}
\|T\|_{\mathcal{L}([A_0, A_1]_{\theta,p}, [B_0, B_1]_{\theta,p}) } \leq  \|T\|_{\mathcal{L}(A_0, B_0)}^{1- \theta}\|T\|_{\mathcal{L}(A_1, B_1)}^{\theta}.
\end{equation}
For a proof of \eqref{eq:operator_interpolation}, see \cite[Proposition 14.1.5]{brenner2008mathematical}. Throughout the paper, the map $T$ will be a solution operator of a boundary or an initial value problem. This means that the space $A_0$ and $A_1$ will be product spaces. Therefore, we will also make use of the fact that interpolation and Cartesian products commute.

The interpolation of Sobolev spaces yields,  see \cite[Theorem 14.2.7]{brenner2008mathematical}, 
\begin{equation}
[H^m(\Omega), H^k(\Omega)]_{\theta,2} = H^{(1-\theta) m + \theta k}(\Omega), \quad m, k \in [0,\infty), 
\label{eq:interp_Hilbert}
\end{equation}
with equivalent norms. In addition, from the fact that duality and interpolation commute (see \cite[eq. 14.1.8]{brenner2008mathematical}) and \cite[Lemma A2 and Lemma 2.1]{guermond2009lbb}, it follows that  
\begin{align}
[H^1_0(\Omega)^*, (H^2(\Omega)\cap H^1_0(\Omega))^*]_{1/2,2} &= (H^{3/2}(\Omega)\cap H^1_0(\Omega))^*, \label{eq:interp_dual_zero_trace}\\ 
[L^2_0(\Omega), (H^1(\Omega) \cap L^2_0(\Omega))^*]_{1/2,2} &= (H^{1/2}(\Omega) \cap L^2_0(\Omega))^*. \label{eq:interp_dual_zero_avg}
\end{align}
In the above, the space $L^2_0(\Omega) =  \{q \in L^2(\Omega), \,\, \int_{\Omega} q = 0\}$. For interpolation results on the boundary, we only use that 
\begin{equation}
[H^{1/2}(\partial \Omega), H^{1/2} (\partial \Omega)^*]_{1/2,2} = L^2(\partial \Omega), \label{eq:interpolation_dual}
\end{equation}
which holds from an abstract theorem on interpolating between a space and its dual, see \cite{cobos1998theorem} or see  \cite[Theorem B.11]{mclean2000strongly}  for a more general statement.  For interpolation results in Bochner spaces, we use the following result, see \cite[Theorem 2.2.10]{hytonen2016analysis}. For any Banach spaces $B_1 \subset B_0$ and $\theta \in (0,1)$,  
\begin{equation}
[L^2(I , B_0), L^2(I, B_1)]_{\theta,2} = L^2(I, [B_0,B_1]_{\theta, 2} ). \label{eq:interpolation_bochner}
\end{equation}

We now  fix notation with regards to neural networks.
Consider natural numbers \(d, m, L, N_0, \dots, N_L\) and let $\theta = \left((A_1, b_1), \dots, (A_L, b_L)\right)$ be a collection of matrix-vector pairs with $A_l\in\mathbb R^{N_{l}\times N_{l-1}} $, $b_l\in\mathbb R^{N_l}$ and $N_0 = d, N_L = m$. Every matrix-vector pair $(A_l, b_l)$ induces an affine linear map \(T_l\colon \mathbb R^{N_{l-1}} \to\mathbb R^{N_l}\) via $x\mapsto A_lx + b_l$. The \emph{neural network function with parameters} \(\theta\) and with respect to an \emph{activation function} \(\rho\colon\mathbb R\to\mathbb R\) is the function 
\begin{equation}
      u_\theta\colon\mathbb R^d\to\mathbb R^m, \quad x\mapsto T_L(\rho(T_{L-1}(\rho(\cdots \rho(T_1(x)))))).  \label{eq:neural_function}
\end{equation}
The depth of a network is $L$, and the width of the $\ell$--th layer is $N_{\ell}$. We call a network \emph{shallow} if it has depth \(2\) and \emph{deep} otherwise. With $\Theta$ we denote the parameter space of a fixed architecture, it holds that $\Theta = \mathbb R^p$ for a suitable $p\in\mathbb N$. The collection of all neural network functions with a certain parameter space $\Theta$ is denoted by $\mathcal F_\Theta$. Finally, we remark that for smooth activation functions, such as $ 
\rho = \tanh$, it holds $\mathcal F_\Theta \subset  H^k(\Omega)^m $ for all $k\in\mathbb N$.

We denote by $C$ a generic constant independent of the network parameters $\theta$. We use the standard notation $A \lesssim B$ for $A \leq C \, B$.

\section{Abstract Framework} \label{sec:abstract_framework} 
We discuss an abstract framework to derive a C\'ea Lemma for physics-informed neural networks. There are two key differences form the finite element case that need to be taken into account. First, a neural network ansatz class is not a sub-vectorspace. Thus, Galerkin orthogonality does not hold -- this renders the proof of the classical C\'ea Lemma invalid in this case. Second, one can not expect to find global minimizers when employing neural network ansatz functions. Both aspects are addressed in this Section. The provided C\'ea Lemma is from \cite{muller2022error}, however the straight-forward applicability to PINNs was not realized in that paper. The main focus is on linear stationary and evolutionary PDEs. 

\subsection{Linear PDEs}
We focus on linear PDEs that are posed in Hilbert spaces.

\begin{setting}[Abstract Setting]\label{setting:linear_equations}
    Assume we are given two Hilbert spaces, $X$ and $Y$ and a bounded linear map 
    \begin{equation*}
        T:X \to Y, \quad u\mapsto Tu.
    \end{equation*}
    For a fixed $f\in Y$ we aim to solve the equation 
    \begin{equation}\label{eq:abstract_pde}
        Tu = f.
    \end{equation}
    We will always assume that there exists a unique solution $u^*$.\footnote{The requirement of uniqueness can be relaxed, but for simplicity we do not comment on it further.} Correspondingly, we define the least-squares minimization problem as 
    \begin{equation}\label{eq:abstract_pinn_energy}
        E:X \to \mathbb{R}, \quad E(u) = \frac{1}{2}\lVert Tu - f \rVert^2_{Y}.
    \end{equation}
    In order to approximately solve equation \eqref{eq:abstract_pde}, we minimize the energy $E$ over the ansatz class $\mathcal{F}_
    \Theta$. In practice, 
    the above energy is discretized. Namely, we minimise a loss functional 
    $$L: \Theta \rightarrow \mathbb R, \quad L(
    \theta) \approx E(u_\theta). $$ This discrete  loss function is assumed to be computable, and to be a good approximation of the energy $E$ (more precise details will appear later).  This loss function typically results from discretizing the $L^2$-norms, see example \ref{example:forward}. Then, the PINN solution is obtained by approximately solving the following optimization problem. 
    \begin{equation}
    \min_{ \theta \in \Theta} L(\theta). \label{eq:optim_problem_disc}
    \end{equation} 
We denote by  $\theta^* \in \Theta$ the inexact solution given by an optimization procedure applied to 
\eqref{eq:optim_problem_disc}. Note that $\theta^*$ is typically not the global minimizer of $L$. The corresponding neural network function, see \eqref{eq:neural_function}, is denoted by $u_{\theta^*}$  
\end{setting}
The above framework can be used for both forward and certain inverse problems. Typically $Y$ is a product space and consequently $T$ has several components which correspond to the PDE term and to  initial and boundary conditions.

\begin{example}[Forward Problem]\label{example:forward}
    As an example on a domain $\Omega\subset \mathbb R^d$, consider Poisson's equation
    \begin{align*}
        -\Delta u &= f \quad \mathrm{in }\,\, \Omega, 
        \\
        u &= g \quad \mathrm{on }\,\, \partial\Omega, 
    \end{align*}
     for a forcing $f\in L^2(\Omega)$. The corresponding spaces are $X=H^2(\Omega)$, $Y=L^2(\Omega)\times L^2(\partial\Omega)$. The map $T$ is given by
     \begin{equation*}
         T:H^2(\Omega) \to L^2(\Omega)\times L^2(\partial\Omega), \quad Tu = (-\Delta u, u|_{\partial\Omega}).
     \end{equation*}
     The associated least-squares energy is 
     \begin{equation*}
         E:H^2(\Omega) \to \mathbb R, \quad E(u) 
         =
         \frac12 \lVert \Delta u + f \rVert_{L^2(\Omega)}^2 
         +
         \frac12 \lVert u - g \rVert_{L^2(\partial\Omega)}^2.
     \end{equation*}
     To discretize, one can choose collocation points $\{x_i\}_{i=1,\dots,N_\Omega}$ in $\Omega$ and $\{x^b_i\}_{i=1,\dots,N_{\partial\Omega}}$ on the boundary $\partial\Omega$ and define
     \begin{equation*}
         L(\theta)
         =
         \frac{|\Omega|}{2N_{\Omega}}
         \sum_{i=1}^{N_{\Omega}}(\Delta u_\theta(x_i) + f(x_i))^2
         +
         \frac{|\partial\Omega|}{2N_{\partial\Omega}}
         \sum_{i=1}^{N_{\partial\Omega}}(u_\theta(x^b_i) - g(x^b_i))^2.
     \end{equation*}
\end{example}

\begin{example}[Inverse Problem]
    We discuss how to treat a source recovery problem for the Laplacian in the context of our abstract framework. Let $\eta \geq 0$ denote a regularization parameter and $u_d\in L^2(\Omega)$ an interior observation. Denote by $\mathcal H$ the closure of $C^\infty(\overline{\Omega})$ with respect to the norm $\|u\|_{\mathcal H} = \|\Delta u\|_{L^2(\Omega)} + \|u\|_{L^2(\partial\Omega)}$.\footnote{This space is analyzed in detail later, in Section~\ref{sec:inverse_problem}.} Then, we define
    \begin{equation*}
        T:\mathcal H\times L^2(\Omega) \to L^2(\Omega)^3\times L^2(\partial\Omega) = Y, \quad T(u,f) = (f+\Delta u, u, \eta f, u|_{\partial\Omega})
    \end{equation*}
    As a right-hand side, we choose $(0,u_d, 0,g)$ which yields
    \begin{align*}
        E(u,f) 
        &=
        \frac12 \| T(u,f) - (0,u_d,0,g) \|^2_Y
        \\
        &=
        \frac12 \| f + \Delta u \|^2_{L^2(\Omega)}
        +
        \frac12 \| u - u_d \|^2_{L^2(\Omega)}
        +
        \frac12 \| u - g \|^2_{L^2(\partial\Omega)}
        +
        \frac{\eta^2}{2}\|f\|^2_{L^2(\Omega)}.
    \end{align*}
    This corresponds to the typical formulation that is used in the PINN community to formulate inverse/PDE constraint optimization problems.
\end{example}

To derive error estimates, the crucial ingredient is the coercivity of the energy $E$, which can be deduced from properties of the operator $T$.  Assume we are in Setting \ref{setting:linear_equations}. We define the following form
    \begin{equation} \label{eq:form_a_general}
        a:X\times X \to \mathbb R, \quad a(u,v) = \langle Tu, Tv \rangle_Y,
    \end{equation}
    where we utilize the inner product of $Y$.  
\begin{assumption}
\label{assumption:coercivity} We will make use of the following assumptions. 
\begin{itemize}
    \item[(A1)] The form $a$ is coercive, possibly with respect to a weaker (semi) norm on $X$ which we denote by $||| \cdot |||_X$. There exists an $\alpha > 0$ such that for all $u 
    \in X$, 
    \begin{equation}
        a(u, u) \geq \alpha |||u|||_X^2. \label{eq:coercivity_assumption}
    \end{equation} 
    \item[(A2)]  The operator $T$ is bounded. That is, we have that the form $a$ is bounded with 
    \begin{equation}
        a(u,v) \leq \| T \|^2 \, \|u\|_X \|v\|_X. \label{eq:a_bounded}
    \end{equation}   
    \item[(A3)] The quadrature error is controlled.  Fix an arbitrary $C >0$.  For any $\eta > 0$, assume that  we can have $L(\theta) \approx E(u_\theta)$, such that 
    \begin{equation}
        \sup_{|\theta|_{\infty} \leq C } |E(u_\theta) - L(\theta) | \leq \eta .  
    \end{equation}
    In the case of Monte--Carlo integration, the above estimate holds with high probability. 
    \item[(A4)] The optimization error is controlled. For a given $\epsilon > 0$, it holds that 
    \begin{equation}
         L(\theta^*) - \inf_{\psi \in \Theta} L(\psi)  \leq \epsilon
    \end{equation}
\end{itemize} 
\end{assumption}
\begin{remark}[Validity of the assumptions ] 
 As we will see in Theorems \ref{thm:aposteriori_thm} and \ref{thm:cea}, the coercivity properties of the form $a$ determine the norm in which error decay can be expected. Except for the elliptic case, coercivity properties for PINNs are not studied in the existing literature and we will therefore focus on proving (A1) for concrete applications.    
 Continuity properties are straight-forward to establish as will be demonstrated in the remainder of the article.

 An extensive analysis of the quadrature error for all the concrete applications in this paper can be obtained with Rademacher complexity for Monte--Carlo methods or with classical analysis for numerical  quadrature.  A detailed exposition is beyond the scope of this work. We give an example on how to control the quadrature error in Remark \ref{remark:quad_error} below for the Poisson equation. 
 
 The optimization error is difficult to quantify. However, for greedy algorithms, certain bounds can be proved \cite[Section 5]{siegel2023greedy}. 
\end{remark}

\begin{remark}[Bounds on the quadrature error]\label{remark:quad_error} 
For the Monte-Carlo method, where the collocation points are chosen to be identically and independently distributed (i.i.d.) in the domain and on the boundary, tools from Rademacher complexity can be used to obtain uniform bounds on $\eta$. For more details, we refer to \cite[Section 4]{hu2023solving} for deep neural networks and \cite[Section 7]{siegel2023greedy} for shallow networks. As an illustration, consider example  \ref{example:forward} with $\mathcal{F}_\Theta$ consisting of deep neural networks  and assume that  $\{x_i\}_{i=1,\dots,N_\Omega}$ in $\Omega$ and $\{x^b_i\}_{i=1,\dots,N_{\partial\Omega}}$  in $\partial \Omega$ are i.i.d uniformly. Then, following the proof of \cite[Lemma 4.7]{hu2023solving}, one obtains that with probability $1-\epsilon$ for any  $\epsilon \in (0,1)$, 
\begin{equation}
|\eta| \leq \sup_{\theta \in \Theta} |E(u_\theta) - L(\theta) | \leq C_{\Theta} \left( N_{\Omega}^{-1/2}(1+\log^{1/2} (N_\Omega) ) +  N_{\partial \Omega}^{-1/2}(1+\log^{1/2} (N_{\partial \Omega}) )  \right). 
\end{equation}
The constant $C_\Theta$ depends on the ansatz class $\mathcal{F}_{\Theta}$ ( $\mathrm{dim}(\Theta)= p, \sup_{\theta \in \Theta} |\theta|_{l^{\infty}}$, and on the activation function $\rho$), on $\|f\|_{L^{\infty}(\Omega)}$, on $\|g\|_{L^{\infty}(\partial \Omega)}$, and on $\log^{1/2} \epsilon^{-1}$ but not on $N_\Omega$ and $N_{\partial \Omega}$. For shallow neural networks, the dependence on $\log^{1/2}(N_\Omega)$ and $\log^{1/2}(N_{\partial \Omega})$ can be removed, see \cite[Theorem 6]{siegel2023greedy}.   We also  refer to \cite[Section 7.3]{siegel2023greedy} for error estimates when numerical quadrature is used.  
\end{remark}
We now present the a posteriori error  estimate.
\begin{theorem}[A Posteriori Error Control]\label{thm:aposteriori_thm}
    Assume we are in Setting \ref{setting:linear_equations} and Assumption (A1) holds. Let $\mathcal F_\Theta$ be a subset of $X$. For an arbitrary element $u_{\theta^*}\in \mathcal F_\Theta$ we can estimate
    \begin{equation} \label{eq:aposteriori_general}
        ||| u_{\theta^*} - u^* |||_X^2
        \leq 
        \frac{2}{\alpha} ( L(\theta^*) + \eta(\theta^*)) ,
    \end{equation}
    where $\eta(\theta^*)$ represents a local quadrature error $\eta(\theta^*) =  E(u_{\theta^*}) - L(\theta^*)$. If (A3) holds, then we can estimate $\eta(\theta^*) \leq \eta. $
\end{theorem}
\begin{proof}
We observe that \begin{equation}\label{eq:energy_scaling}
        E(u_{\theta^*}) - E(u^*) 
        = \frac12 a(u_{\theta^*} - u^*, u_{\theta^*} - u^*).  
    \end{equation}
    which holds by expanding the right hand side and using  that $u^*$ solves the PDE. With the coercivity property \eqref{eq:coercivity_assumption}, we obtain 
    \begin{equation*}
        ||| u_{\theta^*} - u^* |||^2_X 
        \leq
        \frac{1}{\alpha} a(u_{\theta^*} - u^*, u_{\theta^*} - u^*)
        = 
        \frac{2}{\alpha}E(u_{\theta^*})
        =
        \frac{2}{\alpha}(L(\theta^*) + \eta(\theta^*)).
    \end{equation*}
    This completes the proof.
\end{proof}
We now present a C\'ea's type estimate.
\begin{theorem}[C\'ea's Lemma]\label{thm:cea}
    Assume we are in Setting \ref{setting:linear_equations} and assumptions (A1)-(A2) hold. Let $\mathcal F_\Theta$ be a subset of $X$. For an arbitrary element $u_{\theta^*}\in \mathcal F_\Theta$ we can estimate
    \begin{equation}
        ||| u_{\theta^*} - u^* |||_X \leq \sqrt{\frac{2 \, \delta}{\alpha} + \frac{ \|T \|^2}{\alpha}\inf_{u_\psi\in \mathcal F_\Theta}\|u_\psi - u^*\|^2_X} \,\, , \label{eq:main_thm_estimate}
    \end{equation}
    where \begin{equation}
        \delta = \delta(u_{\theta^*}) = E(u_{\theta^*}) - \inf_{u_\psi\in \mathcal F_\Theta}E(u_\psi). \label{eq:delta_opt}
    \end{equation}
     Further, we have the following bound on $\delta$.
 \begin{align} \label{eq:explicit_bound_delta}
|\delta(u_{\theta^*})| \leq  \sup_{\theta \in \Theta }|  E(u_\theta) - L(\theta) |  + (L(\theta^*) - \inf_{\psi \in \Theta} L(\psi) ).  
\end{align}   
 In addition, if $\Theta = \{\theta \in \mathbb{R}^p, \,\, |\theta|_{\infty} \leq C\}$ for some $C >0$ and if assumptions (A3) and (A4) hold, then $\delta$ can be a priori controlled by the quadrature and the optimization error, i.e. $
    \delta  \leq  \eta + \epsilon.$ 
    \end{theorem}
\begin{proof}
    The proof is due to \cite{muller2022error} for a slightly different setting. With coercivity \eqref{eq:coercivity_assumption} and \eqref{eq:energy_scaling}, we have 
    \begin{align*}
        ||| u_{\theta^*} - u^* |||^2_X 
        &\leq
        \frac{1}{\alpha} a(u_{\theta^*} - u^*, u_{\theta^*} - u^*)
        \\
        & = 
        \frac{2}{\alpha}(E(u_{\theta^*}) - E(u^*))
        \\
        &=
        \frac{2}{\alpha}(E(u_{\theta^*}) - \inf_{u_\psi\in \mathcal F_\Theta }E(u_\psi)) 
        +
        \frac{2}{\alpha}\inf_{u_\psi\in \mathcal F_\Theta }\left[ E(u_\psi) - E(u^*) \right]
        \\
        &\leq
        \frac{2\delta}{\alpha} + \frac{ \|T \|^2}{\alpha}\inf_{u_\psi\in \mathcal F_\Theta }\|u_\psi - u^*\|^2_X. 
    \end{align*}
    To obtain the last bound above, we write $2(E(u_\psi) - E(u^*)) = a(u_\psi - u^*, u_\psi - u^*)$,  and we use \eqref{eq:a_bounded}. To show \eqref{eq:explicit_bound_delta}, 
we write 
\begin{equation}
     \label{eq:expanding_delta}
     \delta(u_{\theta^*}) 
     = 
     (E(u_{\theta^*}) - L(\theta^*)) 
     +
     (L(\theta^*) - \inf_{\theta \in \Theta} L(\theta) ) 
     \\  
     + (\inf_{\theta \in \Theta} L(\theta) 
     -
     \inf_{u_\theta \in 
    \mathcal{F}_\Theta} E(u_\theta)). 
\end{equation}
Noting that 
the first and third terms in \eqref{eq:expanding_delta} can be bounded by $\sup_{\theta \in \Theta }| L(\theta) - E(u_\theta) | $  yields 
\eqref{eq:explicit_bound_delta}.   
\end{proof}
\subsection{Approximation results} Here, we briefly present bounds on the infimum term in \eqref{eq:main_thm_estimate} in useful special cases. 
\begin{theorem}\label{thm:quant_approx}
	    Let $\Omega \subset \mathbb{R}^d$, $d\in \mathbb{N}$, be a bounded Lipschitz domain. Moreover, let $k , m \in\mathbb{N}_0$ with $k>m$. Then, for every $n\in\mathbb{N}$ and every $u\in H^{k}(\Omega)$, there exists a fully-connected $\tanh$-network $u_{\theta_n}\in H^{m}(\Omega)$ with parameter space $\Theta_n$ of dimension $\mathcal{O}(n)$ such that it holds
	    \begin{equation}\label{eq:estimate_sobolev}
	        \lVert u - u_{\theta_n} \rVert_{H^m(\Omega)} \lesssim \, \left( \frac1n \right)^{\frac{k-m-\mu}{d}}\lVert u \rVert_{H^{k}(\Omega)},
	    \end{equation}
     where $\mu>0$ can be arbitrarily small.
\end{theorem}

In case of Barron regularity, the dependence on the dimension in the above rate can be removed. The spectral Barron norm is defined as a weighted $L^1$ norm of the Fourier transform, more precisely 
\begin{equation*}
    \mathcal B^m(\Omega) 
    =
    \left\{
    u\in L^2(\Omega) \mid \inf_{u_{E|\Omega = u}}\int_{\mathbb R^d}(1 + |\omega|)^s |\hat u_E|\,\mathrm d\omega < \infty
    \right\}
\end{equation*}
where $\hat u$ denotes the Fourier transform of a function $u$, and $u_E$ denotes the extension of a $L^2(\Omega)$ function to all of $\mathbb R^d$, see \cite{siegel2020approximation}. 

\begin{theorem}\label{thm:quant_approx_barron}
	    Let $\Omega \subset \mathbb{R}^d$, $d\in \mathbb{N}$, be a bounded Lipschitz domain. Let $m,n \in\mathbb{N}$. Denote by $\mathcal F_{\Theta_n}$ the set of shallow neural networks of width $n$ with $\tanh$ activation function. Then it holds
	    \begin{equation} \label{eq:estimate_barron}
	    \inf_{u_\theta \in \mathcal{F}_{\Theta_n}}    \lVert u - u_{\theta} \rVert_{H^m(\Omega)} \leq  C(\mathrm{diam}(\Omega), m) |\Omega|^{1/2} \,\,\left(\frac{1}{n}\right)^{\frac12} \lVert u \rVert_{\mathcal{B}^{m+1}(\Omega)}.
  	    \end{equation}
\end{theorem}
\begin{remark}
Theorem~\ref{thm:quant_approx} and Theorem~\ref{thm:quant_approx_barron} are special cases of \cite[Theorem 4.9]{GR20} and \cite[Corollary 1]{siegel2020approximation} respectively. The results hold for most commonly used activation functions, we refer the reader to Table 1 in \cite{GR20} and the Table below \cite[Corollary 1]{siegel2020approximation}.
\end{remark}

\begin{remark}
    Theorem~\ref{thm:quant_approx_barron}  is one example of the approximation theory in Barron spaces, see \cite{siegel2022high} for improved rates for $\mathrm{ReLU}^k$ activation functions. There are some results on the relation between Barron and 
    Sobolev spaces, see for instance \cite[Lemma 2.5]{xu2020finite}.  
 \end{remark}
\section{Applications of the theory}\label{sec:applications} In this section, we consider several applications of the above abstract setting. Namely, we prove error estimates for PINN formulations of the Poisson, Darcy, elasticity, Stokes,  parabolic, and hyperbolic equations. Here, we only explicitly present a priori estimates of the type \eqref{eq:main_thm_estimate}. A posteriori estimates follow immediately from the coercivity results that we prove, see Theorem~\ref{thm:aposteriori_thm}. We omit their write up for brevity.

\subsection{Poisson's Equation} We consider the following boundary value problem 
\begin{subequations}
\begin{alignat}{2}
    - \nabla \cdot (A \nabla u) & = f, && \quad  \mathrm{in} \,\, \Omega \label{eq:pimal_poisson_1} \\ 
    u & = g_D,  && \quad  \mathrm{on} \,\, \partial \Omega.  \label{eq:pimal_poisson_2}
\end{alignat}
\end{subequations}
We assume that $f \in L^2(\Omega)$, $g_D \in H^{3/2}(\partial \Omega)$, and  $A$ is an elliptic and smooth  tensor with ellipticity constant $a_0 >0$. Furhter, we assume $\partial \Omega \in C^{1,1}$.  It therefore follows that $u \in H^2(\Omega)$.  In this case, the residual energy $E: H^2(\Omega) \to \R$ is defined as: 
\begin{equation}
     E(u) = \frac12 \|\nabla \cdot (A \nabla u) - f \|^2_{L^2(\Omega)} \\ + \frac{1}{2} \|u - g_D\|^2_{L^2(\partial \Omega)}. \label{eq:E_poisson}
\end{equation}
The form $a$, given in \eqref{eq:form_a_general}, reads in this setting as $a :  H^2(\Omega) \times H^2(\Omega) \to \R$:
\begin{equation}
    a(u,v) = (\nabla \cdot (A \nabla u) , \nabla \cdot (A \nabla v) )_{\Omega} +  (u,v)_{\partial \Omega}.   \label{eq:a_poisson}
\end{equation}
The global bound in Theorem \ref{thm:error_estimate_poisson} below   is presented in \cite{muller2022notes, shin2020convergence}. Here,  we provide a proof of the coercivity estimate used in \cite{muller2022notes,shin2020convergence}, and we show error estimates with additional observations.
\begin{lemma}[Coercivity] \label{lemma:coercivity_poisson}
The form $a$ given in \eqref{eq:a_poisson} is coercive with respect to the $H^{1/2}(\Omega)$ norm. There exists a constant $C_1 > 0 $ such that 
\begin{align}
    C_1 \|u\|_{H^{1/2}(\Omega)}^2 \leq   a(u,u). \label{eq:coerc_global_poisson}
\end{align}
In addition, for any set $U$ compactly supported in $\Omega$, there holds for $C_2 >0 $ 
\begin{equation}
   C_2  \|u\|^2_{H^2(U)} \leq a(u,u).  \label{eq:coerc_interior_poisson}
\end{equation}
\end{lemma}
\begin{proof}
Estimate \eqref{eq:coerc_global_poisson} results from regularity theory and can be found in \cite{schechter1963p,bramble1970rayleigh}. Here, we outline the proof since the strategy will be useful later. 
From standard regularity theory, we have
\begin{equation}
    \|u\|_{H^2(\Omega)} \lesssim \|\nabla \cdot (A\nabla u )\|_{L^2(\Omega)} + \|u\|_{H^{3/2}(\partial \Omega)}. 
\end{equation}
From this estimate, duality arguments are employed (see Appendix \ref{sec:poisson_negative} for details)  to show that 
\begin{equation}
\|u\|_{L^2(\Omega)} \lesssim \|\nabla \cdot (A \nabla u )\|_{(H^{2}(\Omega) \cap H_0^1(\Omega))^*} + \|u\|_{H^{1/2}(\partial \Omega)^*}. 
\end{equation}
In addition, from standard applications of Lax-Milgram for the weak formulation of a corresponding  lifted problem, see for e.g \cite[Proposition 2.10] {ern2004theory}, we obtain that 
\begin{equation}
\|u\|_{H^1(\Omega)} \lesssim \|\nabla \cdot (A \nabla u) \|_{H_0^1(\Omega)^*} + \|u\|_{H^{1/2}(\partial \Omega)}.     
\end{equation}
By interpolation theory, namely by \eqref{eq:operator_interpolation}, \eqref{eq:interp_Hilbert}, 
\eqref{eq:interp_dual_zero_avg}, and \eqref{eq:interpolation_dual}, we obtain that  
 for all $u \in H^2(\Omega)$, 
\begin{multline}
\|u\|^2_{H^{1/2}(\Omega)} \lesssim (\|\nabla \cdot(A \nabla u)\|^2_{(H^{3/2}(\Omega) \cap H^1_0(\Omega))^*} +\|u\|^2_{L^2(\partial \Omega)}) \\ \lesssim ( \|\nabla \cdot(A \nabla u)\|^2_{L^2(\Omega)} +\|u\|^2_{L^2(\partial \Omega)}) = C_{\mathrm{reg}} a(u,u).  \label{eq:first_coercivity}
\end{multline}
For \eqref{eq:coerc_interior_poisson}, we use interior regularity estimates, see e.g \cite[Section 6.3.1]{evans2022partial}. 
\begin{equation}
  \|u\|^2_{H^2(U)} \lesssim \|\nabla \cdot (A \nabla u)  \|^2_{L^2(\Omega)} + \|u\|^2_{L^2(\Omega)}  
\lesssim\|\nabla \cdot (A \nabla u)  \|^2_{L^2(\Omega)} + \|u\|^2_{H^{1/2}(\Omega)}    \lesssim  a(u,u). 
\end{equation}
In the above, we used \eqref{eq:first_coercivity}.
\end{proof}
\begin{remark}
The coercivity estimate in \eqref{eq:coerc_global_poisson} is sharp in the sense that 
$$\|u\|_{H^s(\Omega)}^2 \lesssim a(u,u) \,\,\, \mathrm{iff} \,\,\,  s \leq 1/2.$$  The proof follows by contradiction and by defining a harmonic extension of functions in $H^{3/2}(\partial \Omega)$, see \cite[Theorem 8 and eq. 8]{muller2022notes} for more details. 
\end{remark}
\begin{theorem}(Error Estimates) \label{thm:error_estimate_poisson}
   Let $\theta^*$ inexactly  
minimize $L(\theta)$ with $L(\theta) \approx E(u_{\theta})$ where $E$ is given in \eqref{eq:E_poisson}. Let  $u_{\theta^*}$ be the corresponding neural network.  Further let $U$ be any set compactly supported in $\Omega$. 
Then,  we have that  
\begin{equation}
 \|u_{\theta^*} - u\|^2_{H^{1/2}(\Omega)} +    \|u_{\theta^*} - u\|^2_{H^{2}(U)} \lesssim   L(\theta^*) + \eta(\theta^*),  \label{eq:second_estimate_poisson}
\end{equation}
where $\eta(\theta^*) = E(u_\theta^*) - L(\theta^*)$ is a quadrature error. Further,
\begin{equation}
\|u_{\theta^*} - u\|^2_{H^{1/2}(\Omega)} + \|u_{\theta^*} - u\|^2_{H^{2}(U)}   \lesssim \delta(u_{\theta^*}) + \inf_{u_\theta \in \mathcal{F}_\Theta} \|u_{\theta} - u\|^2_{H^2(\Omega)},  \label{eq:first_estimate_poisson}
\end{equation}
where $\delta(u_{\theta^*})$ is given in \eqref{eq:delta_opt}.  
\end{theorem}

\begin{proof}
The results follows from applications of Theorems~\ref{thm:aposteriori_thm} and \ref{thm:cea}. We verify the  continuity of $a$ on $X =H^2(\Omega)$ which easily follows from the uniform boundedness of $A$ and from trace estimates. Indeed, we have that 
\begin{equation}
a(u,v) \leq 2(\|A\|^2 + C_{\mathrm{tr}}^2)\|u\|_{H^2(\Omega)} \|v\|_{H^2(\Omega)}. 
\end{equation}
Here, the norm $\|A\| = \max_{i,j} \|A_{i,j}\|_{C^1(\Omega)}$. To show \eqref{eq:first_estimate_poisson}, we apply Theorem \ref{thm:cea} with $|||\cdot |||_{X} = \|\cdot\|_{H^{1/2}(\Omega)}$ where  the  coercivity estimate \eqref{eq:coerc_global_poisson} is used  and with  $|||\cdot |||_{X} = \|\cdot\|_{H^{2}(U)}$ where we use the  coercivity estimate \eqref{eq:coerc_interior_poisson}. 
To show \eqref{eq:second_estimate_poisson}, similar applications to Therorem \ref{thm:aposteriori_thm}  are utilized. 
\end{proof}
\begin{remark}[Exact Boundary Values] If we assume that our ansatz class $\mathcal F_\Theta$ satisfies the boundary values exactly, i.e., it holds $\mathcal F_\Theta \subset H^2(\Omega) \cap H^1_{g_D}(\Omega)$ the forms $E$ and $a$ given in \eqref{eq:E_poisson} and \eqref{eq:a_poisson} are modified to  
\[ 
    E_0(u) =\frac12 \|\nabla \cdot (A \nabla u) - f \|^2_{L^2(\Omega)},  \quad a_0(u,v) = (\nabla \cdot (A \nabla u) , \nabla \cdot (A \nabla v) )_{\Omega}.
\]
From regularity theory, it immediately follows that the form $a_0$ is coercive with respect to the $H^2$--norm over $H^2(\Omega) \cap H^1_0(\Omega)$. Thus,  error estimates in the $H^2$--norm can be directly deduced from Theorems~\ref{thm:aposteriori_thm} and  ~\ref{thm:cea}. Standard neural networks do not satisfy the inclusion $\mathcal F_\Theta \subset H^2(\Omega) \cap H^1_{g_D}(\Omega)$, however for certain domains it is possible to adapt the network architecture to conform to the boundary conditions, we refer to \cite{sukumar2022exact}.
\end{remark}

\subsection{Darcy's Equations} 
Here, we consider  Darcy's problem for the unknown pressure $p$ and velocity $\bm \sigma$: 
\begin{alignat}{2} \label{eq:Darcy_0}
    A^{-1} \bm \sigma + \nabla p &= \bm{f},  &&  \quad \mathrm{in} \,\, \Omega \\ 
    \nabla \cdot \bm \sigma & = g, && \quad \mathrm{in} \,\, \Omega \\ 
    p & = g_D, && \quad \mathrm{on} \,\, \partial \Omega.  \label{eq:Darcy_end}
\end{alignat}
In the above, $A$ is the symmetric permeability tensor with positive and uniformly bounded eigenvalues in $\Omega$. We also assume that $\bm{f} \in L^2(\Omega)^d$,  $g \in L^2(\Omega)$ and $g_D \in H^{3/2}(\partial \Omega)$. The associated least-squares energy reads as follows. 
\begin{equation} \label{eq:energy_darcy}
    E( \bm{\sigma}, p ) = \frac{1}{2} \|A^{-1} \bm \sigma + \nabla p - \bm{f}\|_{L^2(\Omega)}^2 +  \frac12 \|\nabla \cdot \bm \sigma  - g\|_{L^2(\Omega)}^2 + \frac12\|p - g_{D}\|_{L^2(\partial \Omega) }^2. 
\end{equation}
The PINNs formulation is to inexactly solve: 
\begin{equation}
\min_{(\theta_1, \theta_2) \in \Theta_1 \times 
\Theta_2 } L(\theta_1, \theta_2), \quad L(\theta_1, \theta_2) \approx E((\bm \sigma_{\theta_1}, p_{\theta_2})). 
\end{equation}
Here too, the method of collocation or Monte Carlo can be employed to approximate the norms in $E$ and obtain $L$. 
The specific form $a$ now reads as $a: (H(\mathrm{div}; \Omega) \times  H^2(\Omega))^2  \rightarrow \R$: 
\begin{equation} \label{eq:a_darcy}
a((\bm{\sigma},p), (\bm{\tau},q)) = 
(A^{-1} \bm{\sigma} + \nabla p, A^{-1} \bm{\tau} + \nabla q)_\Omega + (\nabla \cdot  \bm{\sigma}, \nabla \cdot \bm \tau)_{\Omega} + (p,q)_{\partial \Omega}.
\end{equation}
Here, $H(\mathrm{div}; \Omega)$ and its norm are given by \begin{align*}
    H(\mathrm{div}; \Omega) &= \{ \bm{w} \in L^2(\Omega)^d, \, \mathrm{div}(\bm{w}) \in L^2(\Omega) \}, 
    \\ \|\bm{w}\|_{H(\mathrm{div};\Omega)}^2 & =\|\bm{w}\|_{L^2(\Omega)}^2 + \|\mathrm{div}(\bm{w})\|_{L^2(\Omega)}^2.  
\end{align*} 

\begin{lemma}[Coercivity Estimate]\label{lemma:coercivity_darcy}
    Let  $A\in W^{1,\infty}(\Omega)$ be uniformly elliptic and let $\partial \Omega \in C^{1,1}$. Then for all $\sigma \in H^1(\Omega)^d$ and  $p\in H^2(\Omega)$, we may estimate
    \begin{equation}
    \|\bm\sigma \|^2_{H^{1/2}_{00}(\Omega)^* } +      \| p \|^2_{H^{1/2}(\Omega)} 
        \lesssim a((\bm{\sigma},p), (\bm{\sigma},p)), 
    \end{equation}
    where $H^{1/2}_{00}(\Omega) = [L^2(\Omega), H^1_0(\Omega)]_{1/2,2}$ is known as the Lions-Magenes space, see \cite[Section 2.1]{nochetto2015pde} and \cite[Theorem 11.7]{lions2012non}. 
\end{lemma}
\begin{proof}
    We introduce the notation $\tilde{\bm f}  = \bm \sigma + A \nabla p$,  $\tilde{ g}=\nabla\cdot\bm \sigma$, and $\tilde{g}_D = p \vert_{\partial \Omega}$. With noting that  $\nabla \cdot (A \nabla p ) = \nabla \cdot \tilde{\bm{f}} - \tilde g $ and employing  estimate \eqref{eq:first_coercivity},  it holds that
    \begin{align*}
        \| p \|_{H^{1/2}(\Omega)}
        &
        \lesssim
        \| \tilde{ g} \|_{(H^{3/2}(\Omega) \cap H^1_0(\Omega))^*} 
        +
        \|\nabla\cdot \tilde{\bm f} \|_{(H^{3/2}(\Omega) \cap H^1_0(\Omega))^*}
        +
        \|\tilde g_D\|_{L^2(\partial\Omega)}
        \\
        &\lesssim
        \|\tilde{ g}\|_{L^2(\Omega)}
        +
        \|\tilde{\bm f}\|_{L^2(\Omega)}
        +
        \|\tilde g_D\|_{L^2(\partial\Omega)}.
    \end{align*}
    We proceed with $\sigma$. It holds $\sigma = \tilde{\bm f} - A \nabla p$, then we may estimate
    \begin{equation*}
        \|\bm \sigma\|_{_{[H^1_0(\Omega)^*, L^2(\Omega)]_{1/2}}}
        \lesssim
        \|\tilde{\bm f}\|_{L^2(\Omega)}
        +
        \|A \nabla p\|_{_{[H^1_0(\Omega)^*, L^2(\Omega)]_{1/2}}} \overset{*}{\lesssim}
        \|\tilde{\bm f}\|_{L^2(\Omega)}
        +
        \|p\|_{H^{1/2}(\Omega)}.
    \end{equation*}
     This concludes the proof once $(*)$ is established, which follows by interpolation. Indeed, from the definition of $\|\cdot \|_{H^{1}_0(\Omega)^*}$, integration by parts, and the boundedness of $A$,  we have that 
     \[\|A \nabla p\|_{{H^{1}_0(\Omega)^*}} \lesssim \|p\|_{L^2(\Omega)} \]
Interpolating between the above bound and the trivial bound $\|A\nabla  p \|_{L^2(\Omega)} \lesssim \|p\|_{H^1(\Omega)}$ yields the result where we used \eqref{eq:operator_interpolation} and \eqref{eq:interp_Hilbert}.
\end{proof}

\begin{theorem}[Error Estimate]
Assume that the solution $(\bm \sigma, p ) \in H^1(\Omega)^d \times H^2(\Omega)$. Let $\theta^* = (\theta_1^*, \theta_2^*) $ inexactly minimize  $L( \theta) = L(\theta_1, \theta_2) \approx E(\bm \sigma_{\theta_1}, p_{\theta_2})$ where $E$ is given in \eqref{eq:energy_darcy}. Let  $\bm \sigma_{ \theta_1^*}$ and $p_{\theta_2^*}$ be the corresponding neural networks. Then,
\begin{equation}
\|\bm{\sigma}_{\theta_1^*} - \bm{\sigma}\|_{H^{1/2}_{00}(\Omega)^*}^2 + \|p_{\theta_2^*} - p\|_{H^{1/2}(\Omega)}^2 \lesssim L(\theta^*) +  \eta(\theta^*),
\end{equation}
where $\eta(\theta^*) = E(\bm \sigma_{\theta_1^*}, p_{\theta_2^*}) - L(\theta_1^*, \theta_2^*)$ is a quadrature error. Further, 
\begin{multline}
     \|\bm{\sigma}_{\theta_1^*} - \bm{\sigma}\|_{H^{1/2}_{00}(\Omega)^*}^2 + \|p_{\theta_2^*} - p\|_{H^{1/2}(\Omega)}^2 \lesssim \delta(( \bm{\sigma}_{\theta_1^*}, p_{\theta_2^*}))  \\ + \inf_{(\bm \sigma_{\theta_1}, p_{\theta_2} ) \in \mathcal{F}_{\Theta_1} \times \mathcal F_{\Theta_2}} (\|\bm{\sigma}_{\theta_1} - \bm{\sigma} \|^2_{H(\mathrm{div}; 
 \, \Omega) } + \|p_{\theta_2} - p\|_{H^1(\Omega)}^2  ), 
\end{multline}
with $\delta(( \bm{\sigma}_{\theta_1^*}, p_{\theta_2^*})) =E(\bm \sigma_{\theta_1^*}, p_{\theta_2^*}) - \inf_{(\bm \sigma_{\psi_1}, p_{\psi_2} ) \in \mathcal{F}_{\Theta_1} \times \mathcal{F}_{\Theta_2} } E(\bm \sigma_{\psi_1} , p_{\psi_2} )$  as in \eqref{eq:delta_opt}. 
\end{theorem}
\begin{proof}
This follows from Theorems \ref{thm:aposteriori_thm} and \ref{thm:cea} with $$||| (\bm{\tau},q)  |||^2_X =  \|\bm{\tau} \|_{[H^1_0(\Omega)^*, L^2(\Omega)]_{1/2}}^2 +  \|q\|_{H^{1/2}(\Omega)}^2,  $$ and $ \|(\bm{\tau},q)\|_{X}^2 = \|\bm{\tau}\|_{H(\mathrm{div}; \, \Omega)}^2 + \|q\|^2_{H^1(\Omega)}$. Indeed, it is easy to see that $a$ is continuous with respect to $\|\cdot \|_{X}$ norm and the coercivity estimate with respect to $||| \cdot |||_{X}$ is  established in Lemma \ref{lemma:coercivity_darcy}.
\end{proof}

\begin{remark}[Exact Boundary Values] If we assume that our ansatz class exactly satisfies the boundary values, i.e., it holds $\mathcal F_{\Theta_2} \subset H^1_{g_D}(\Omega)$ the forms $E$ and $a$ given in \eqref{eq:energy_darcy} and \eqref{eq:a_darcy} are modified to 
\begin{align*}
    E_0( \bm{\sigma}, p) 
    & = 
    \frac{1}{2} \|A^{-1} \bm \sigma + \nabla p - \bm{f}\|_{L^2(\Omega)}^2 +  \frac12 \|\nabla \cdot \bm \sigma  - g\|_{L^2(\Omega)}^2  
    \\  
    a_0((\bm{\sigma},p), (\bm{\tau},q)) 
    &= 
    (A^{-1} \bm{\sigma} + \nabla p, A^{-1} \bm{\tau} + \nabla q)_\Omega + (\nabla \cdot  \bm{\sigma}, \nabla \cdot \bm \tau)_{\Omega} 
\end{align*}
From \cite[Theorem 2.1]{pehlivanov1994least}, we have the following coercivity result. 
\begin{equation} \nonumber
 a_0((\bm{\sigma},p), (\bm{\sigma},p)) \geq  C (\|\bm{\sigma}\|^2_{H(\dv, \Omega)} + \|p\|_{H^1(\Omega)}^2), 
 \quad \forall (p,\bm{\sigma}) \in  H_0^1(\Omega) \times \bm{H}(\dv; \Omega) .  
\end{equation}
From Theorems~\ref{thm:aposteriori_thm} and \ref{thm:cea},  we can then directly deduce error estimates in the $H^1$--norm for the pressure $p$ and the $H(\mathrm{div};\Omega)$--norm for the velocity $\bm \sigma$. 
\end{remark}

\subsection{Linear Elasticity}\label{sec:elasticity}
We consider the following boundary value problem 
    \begin{alignat}{2}
        - \nabla \cdot (\mathbb C \varepsilon(\bm u)) & = \bm f, && \quad  \mathrm{in} \,\, \Omega \label{eq:elasticity_1} \\ 
        \bm u & = \bm g_D,  && \quad  \mathrm{on} \,\, \partial \Omega.  \label{eq:elasticity_2}
    \end{alignat}
Here, we assume that $\Omega\subset\mathbb R^d$ is a $C^{1,1}$ domain, $\bm f\in L^2(\Omega)^d$, $\bm g_D \in H^{3/2}(\partial\Omega)^d$, $\mathbb C \in C^{0,1}(\Omega, \mathcal L(\mathbb R^{d\times d}_{\textrm{sym}}))$ is symmetric and uniformly elliptic, i.e., it holds $\mathbb C(x)A:A \geq c |A|^2$ and $\mathbb C(x)A:B = \mathbb C(x)B:A$ for all $x\in\Omega$ and $A,B\in\mathbb R^{d\times d}_{\textrm{sym}}$. We use $\varepsilon(\bm u)$ to denote the symmetric gradient, that is $\varepsilon(\bm u) = \frac12(\nabla \bm u + \nabla \bm u^T)$. For isotropic elasticity, $\mathbb C \epsilon = \lambda \mathrm{tr}(\epsilon) I_{d} + 2 \mu \epsilon$ for the Lam\'{e} coefficients $\lambda , \mu$.

The least squares formulation is $E: H^2(\Omega)^d \to \R$ with 
\begin{equation} \label{eq:E_elsasticity}
     E(\bm u) = \frac12 \| \nabla \cdot (\mathbb C \varepsilon(\bm u)) + \bm f \|^2_{L^2(\Omega)} \\ + \frac{1}{2} \|\bm u - \bm g_D\|^2_{L^2(\partial \Omega)}. 
\end{equation}
We recall that the PINNs formualtion is to inexactly solve 
$$ \min_{\theta \in \Theta} L(\theta), \quad L(\theta) \approx E(\bm u_{\theta}).$$
The form $L(\theta)$ can be obtained by collocation or Monte--Carlo. 
The form $a$, given in \eqref{eq:form_a_general}, reads in this setting as $a :  H^2(\Omega)^d \times H^2(\Omega)^d \to \R$:
\begin{equation}
    a(\bm u,\bm v) = (\nabla \cdot (\mathbb C \varepsilon(\bm u)) , \nabla \cdot (\mathbb C \varepsilon(\bm u)) )_{\Omega} +  (\bm u, \bm v)_{\partial \Omega}.   \label{eq:a_elasticity}
\end{equation}

\begin{lemma}[Coercivity]
    The form $a$ given in \eqref{eq:a_elasticity} is coercive with respect to the $H^{1/2}(\Omega)^d$ norm. 
    \begin{align}
        \|\bm u\|_{H^{1/2}(\Omega)}^2 \lesssim a(\bm u,\bm u). \label{eq:coerc_global_elasticity}
    \end{align}
\end{lemma}
\begin{proof}
    From standard regularity theory, we have
\begin{equation}
    \|\bm u\|_{H^2(\Omega)} 
    \lesssim 
    \|\nabla \cdot (\mathbb C \varepsilon(\bm u) )\|_{L^2(\Omega)} 
    +
    \|\bm u\|_{H^{3/2}(\partial \Omega)}. 
\end{equation}
Using this estimate in duality arguments (see Appendix \ref{sec:poisson_negative} for details) we show that 
\begin{equation}
    \|\bm u\|_{L^2(\Omega)} 
    \lesssim 
    \|\nabla \cdot (\mathbb C \varepsilon(\bm u) )\|_{(H^{2}(\Omega) \cap H_0^1(\Omega))^*} 
    +
    \|\bm u\|_{H^{1/2}(\partial \Omega)^*}. 
\end{equation}
Using Lax-Milgram for the corresponding weak formulation we obtain
\begin{equation}
\|\bm u\|_{H^1(\Omega)} \lesssim \|\nabla \cdot (\mathbb C \varepsilon(\bm u)) \|_{H_0^1(\Omega)^*} + \|\bm u\|_{H^{1/2}(\partial \Omega)}.     
\end{equation}
By interpolation theory (\eqref{eq:operator_interpolation}, \eqref{eq:interp_Hilbert}, 
\eqref{eq:interp_dual_zero_avg}, and \eqref{eq:interpolation_dual}), we obtain  
for all $\bm u \in H^2(\Omega)$, 
\begin{align*}
    \|\bm u\|^2_{H^{1/2}(\Omega)} 
    &\lesssim
    (\|\nabla \cdot(\mathbb C \varepsilon(\bm u))\|^2_{(H^{3/2}(\Omega) \cap H^1_0(\Omega))^*} 
    +
    \|\bm u\|^2_{L^2(\partial \Omega)}) 
    \\ 
    &\lesssim 
    ( \|\nabla \cdot(\mathbb C \varepsilon(\bm u))\|^2_{L^2(\Omega)} 
    +
    \|\bm u\|^2_{L^2(\partial \Omega)}) 
    \\
    &=
    a(\bm u,\bm u).
\end{align*}
\end{proof}

\begin{theorem}[Error Estimate]
    Let $\theta^*$ inexactly  
minimize $L(\theta)$ with $L(\theta) \approx E(u_{\theta})$ where $E$ is given in \eqref{eq:E_elsasticity}. Let  $u_{\theta^*}$ be the corresponding neural network. 
The following holds.  
\begin{equation}
  \|\bm u_{\theta^*} - \bm u\|^2_{H^{1/2}(\Omega)} \lesssim L(\theta^*) + \eta(\theta^*), 
\end{equation}
where $\eta(\theta^*) = E(u_{\theta^*}) - L(\theta^*)$ is a quadrature error. Further,
\begin{equation}
    \|\bm u_{\theta^*} - \bm u\|^2_{H^{1/2}(\Omega)} \lesssim \delta( \bm u_{\theta^*}) + \inf_{\bm u_\theta \in \mathcal{F}_\Theta} \|\bm u_{\theta} - \bm u\|^2_{H^2(\Omega) }, \label{eq:estimate_elasticity}
\end{equation}
where $\delta(\bm u_{\theta^*}) = E(\bm u_{\theta^*}) - \inf_{\bm u_\psi \in \mathcal{F}_{\Theta}} E(\bm u_\psi), $ as in \eqref{eq:delta_opt}.
\end{theorem}
\begin{proof}
     We verify the  continuity of $a$ on $X=H^2(\Omega)^d$. Indeed, we have that 
    \begin{equation}
        a(\bm u,\bm v) 
        \leq
        \max(\|\mathbb C\|^2_{C^{0,1}}, C_{\mathrm{tr}}^2)
        \|\bm u\|_{H^2(\Omega)} \|\bm v\|_{H^2(\Omega)}. 
    \end{equation}
    To show  the estimates, we simply apply Theorems \ref{thm:aposteriori_thm} and \ref{thm:cea} with $|||\cdot |||_{X} = \|\cdot\|_{H^{1/2}(\Omega)}$ where  the  coercivity estimate \eqref{eq:coerc_global_elasticity} is used.   
\end{proof}

\subsection{Stokes System} 
For the fluid velocity $\bm{u}$ and pressure $p$, we  consider the following Stokes flow system. 
\begin{alignat}{2}
    - \Delta \bm{u} + \nabla p & = \bm{f}, && \quad  \mathrm{in}  \,\, \Omega, \\ 
    \nabla \cdot \bm{u} & = 0, && \quad \mathrm{in} \,\, \Omega , \\ 
    \bm{u} & = \bm g_D, && \quad \mathrm{ on } \,\,  \partial \Omega. 
\end{alignat}
We assume that $
\bm f \in L^2(\Omega)^d$,   $\bm g_{D} 
\in H^{3/2}(\partial \Omega)^d$ and that $\bm g _D$ satisfies the compatibility condition $\int_{\partial \Omega} \bm g_D \cdot 
\bm n = 0 $. The associated least squares energy is 
\begin{equation}
    \label{eq:energy_stokes}
    E(\bm{u}, p) = \frac12 \|- \Delta \bm{u} + \nabla p -  \bm{f}\|_{L^2(\Omega)}^2 + \frac12 \|\nabla \cdot \bm{u} \|_{L^2(\Omega)}^2 + \frac12 \|\bm{u} - \bm g_D\|_{L^2(\partial \Omega)}^2. 
\end{equation} 
We recall that the PINNs formulation is to inexactly solve: 
\begin{equation}
\min_{(\theta_1, \theta_2) \in \Theta_1 \times 
\Theta_2 } L(\theta_1, \theta_2), \quad L(\theta_1, \theta_2) \approx E(\bm u_{\theta_1}, p_{\theta_2}). 
\end{equation}
Here, the form $a: (H^2(\Omega)^d \times H^1(\Omega) )^2 \rightarrow \R $ is given by:
\begin{align}\label{eq:a_stokes}
a((\bm{u},p), (\bm{\tau}, q)) = (-\Delta \bm{u} + \nabla p,  -\Delta \bm{\tau} + \nabla q)_{\Omega} + (\nabla \cdot \bm{u}, \nabla \cdot \bm{\tau})_{\Omega} + (\bm u , \bm{\tau})_{\partial \Omega}.
\end{align}
\begin{lemma}[Coercivity]\label{lemma:coercivity_stokes}Let $\partial \Omega$ be a $C^2$
boundary \footnote{More generally, $ \Omega$ can be assumed to be $H^2$--regular. Precisely, this means that for every $\bm u\in H^2(\Omega)$ and $p\in H^1(\Omega)\cap L^2_0(\Omega)$, estimate \eqref{eq:stokes_h2_regularity} holds.}. For any  $\bm u\in H^2(\Omega)$ and $p \in H^1(\Omega)$, we can estimate 
\begin{align} \label{eq:coercivity_stokes} 
    \| \bm u \|^2_{H^{1/2}(\Omega)} 
    + 
    \|p\|^2_{(H^{1/2}(\Omega)\cap L^2_0(\Omega))^*} \lesssim         a((\bm{u},p) , (\bm{u} , p)). 
\end{align}
\end{lemma}
\begin{proof}
It suffices to show the result for $u \in H^2(\Omega)$ and $p \in H^1(\Omega) \cap L^2_0(\Omega)$. Indeed, for any $p \in H^1(\Omega)$, let $ \tilde p = p - \langle p \rangle$ where $\langle p \rangle $ is the average of $p$. It is easy to verify that  $\tilde p \in  H^1(\Omega) \cap L^2_0(\Omega)$ with  
\[ \|p\|^2_{(H^{1/2}(\Omega)\cap L^2_0(\Omega))^*} = \|\tilde p\|^2_{(H^{1/2}(\Omega)\cap L^2_0(\Omega))^*}, \,\,\,  a((\bm{u},p) , (\bm{u} , p)) = a((\bm{u},\tilde p) , (\bm{u} , \tilde p)). \]
Hereinafter, $p \in H^1(\Omega) \cap L^2_0(\Omega)$.  Regularity theory for Stokes' equation asserts that \cite[Proposition 2.2]{temam2001navier} \begin{equation}\label{eq:stokes_h2_regularity}
        \|\bm u\|_{H^2(\Omega)} 
        +
        \|p\|_{H^1(\Omega)} 
        \lesssim 
        \|-\Delta \bm u 
        + 
        \nabla p\|_{L^2(\Omega)} 
        + 
        \|\nabla\cdot \bm u\|_{H^1(\Omega)} + \|\bm{u}\|_{H^{3/2}(\partial \Omega)}.
    \end{equation}
    Employing the above regularity estimate, we  consider a dual problem and derive the corresponding duality estimate (see Section~\ref{sec:regularity_stokes} in the Appendix). 
    \begin{align}\label{eq:stokes_dual_h2_regularity}
        \begin{split}
            \|\bm u\|_{L^2(\Omega)} 
            +
            \| p \|_{(H^1(\Omega)\cap L^2_0(\Omega))^*}
            &\lesssim
            \|-\Delta \bm u + \nabla p\|_{(H^2(\Omega)\cap H^1_0(\Omega))^*}
            \\
            &
            +
            \|\nabla \cdot \bm u\|_{(H^1(\Omega)\cap L^2_0(\Omega))^*} +
            \|\bm u\|_{H^{1/2}(\partial\Omega)^*}.
        \end{split}
    \end{align}
    The weak existence theory \cite[Theorem 10]{amrouche1991existence} provides the estimate \begin{align}\label{eq:stokes_h1_regularity}
        \begin{split}
            \|\bm u\|_{H^1(\Omega)}
            +
            \|p\|_{L^2(\Omega)}
            &\lesssim
            \|-\Delta \bm u + \nabla p \|_{H^1_0(\Omega)^*}
            +
            \| \nabla\cdot \bm u \|_{L^2(\Omega)}
            \\ &+ 
            \|\bm u\|_{H^{1/2}(\partial\Omega)},        
        \end{split}
    \end{align}
    where the Stokes operator is to be understood in weak form, i.e.,
    \[
    \langle -\Delta\bm u + \nabla p, \bm\varphi\rangle_{H^1_0(\Omega)} 
    =
    \int_\Omega \nabla\bm u \cdot \nabla \bm\varphi + p \nabla\cdot \bm\varphi\mathrm dx. 
    \]
    Finally, interpolating between \eqref{eq:stokes_dual_h2_regularity} and \eqref{eq:stokes_h1_regularity} we obtain
    \begin{align*}
        \begin{split}
            \|\bm u\|_{H^{1/2}(\Omega)}
            +
            \|p\|_{(H^{1/2}(\Omega)\cap L^2_0(\Omega))^*}
            &\lesssim
            \| -\Delta \bm u + \nabla p \|_{(H^{3/2}(\Omega)\cap H^1_0(\Omega))^*}
            \\
            &+
            \|\bm u\|_{L^2(\partial\Omega)}
            +
            \|\nabla\cdot \bm u\|_{[(H^1(\Omega) \cap L^2_0(\Omega))^* , L^2(\Omega) ]_{1/2,2}}.
        \end{split}
    \end{align*}
    The assertion follows readily from the estimate above and the proof is completed.
\end{proof}
\begin{theorem}[Error Estimates]\label{thm:stokes}
 Let $\theta^* = (\theta_1^*, \theta_2^*) $ inexactly minimize the loss function    $L( \theta) = L(\theta_1, \theta_2) \approx E(\bm u_{\theta_1}, p_{\theta_2})$ where $E$ is given in \eqref{eq:energy_stokes}. Let  $\bm u_{ \theta_1^*}$ and $p_{\theta_2^*}$ be the corresponding neural networks. Then, 
 \begin{equation}
 \|\bm u - \bm{u}_{\theta_1^*} \|^2_{H^{1/2}(\Omega)}
     +
\|p - p_{\theta_2^*}\|^2_{(H^{1/2}(\Omega)\cap L^2_0(\Omega))^*}   \lesssim  
L(\theta^*) + \eta(\theta^*),
 \end{equation}
 where $\eta(\theta^*)  = E(u_{\theta_1^*}, p_{\theta_2^*}) - L(\theta_1^*, \theta_2^*)$. Further,
\begin{multline} \label{eq:error_estimate_stokes}
     \|\bm u - \bm{u}_{\theta_1^*} \|^2_{H^{1/2}(\Omega)}
            +
            \|p - p_{\theta_2^*}\|^2_{(H^{1/2}(\Omega)\cap L^2_0(\Omega))^*}   \lesssim \delta((\bm{u}_{\theta_1^*}, p_{\theta_2^*})) \\ 
            +   \inf_{(\bm{\sigma}_{\theta_1}, q_{\theta_2}) \in \mathcal{F}_{\Theta_1} \times \mathcal{F}_{\Theta_2} } (\|\bm u - \bm{\sigma}_{\theta_1} \|^2_{H^{2}(\Omega)}
            +
            \|p - q_{\theta_2}\|^2_{H^1(\Omega)}), 
\end{multline}
with $\delta((\bm{u}_{\theta_1^*}, p_{\theta_2^*})) = E((\bm{u}_{\theta_1^*}, p_{\theta_2^*})) - \inf_{(\bm u_{\psi_1}, p_{\psi_2}) \in \mathcal{F}_{\Theta_1} \times \mathcal{F}_{\Theta_2}} E((\bm u_{\psi_1}, p_{\psi_2}))$, as in \eqref{eq:delta_opt}.
\end{theorem}
\begin{proof}
The proof makes use of the abstract Theorem \ref{thm:cea} with $||| (\bm \sigma, q )|||_{X}^2 = \|\bm{\sigma}\|_{H^{1/2}(\Omega)}^2 + \|q\|_{(H^{1/2} \cap L^2_0(\Omega))^*}^2$ and $\|(\bm \sigma, q )\|_{X}^2 = \|\bm{\sigma}\|_{H^{2}(\Omega)}^2 + \|q\|_{H^{1}(\Omega)}^2$. The continuity assumption is easy to verify and the coercivity requirement  \eqref{eq:coercivity_assumption} is established in Lemma \ref{lemma:coercivity_stokes}.
\end{proof}
\begin{remark}[Average Penalty]
For a unique pressure solution, we  let $p \in L^2_0(\Omega)$ and penalize the pressure average $\langle p \rangle$. The forms $E$ and $a$ given in \eqref{eq:energy_stokes} and \eqref{eq:a_stokes} are  modified to 
\begin{align}
E_{\int}(\bm u, p) & = E(\bm u, p) + \frac12 \|\langle p \rangle\|^2_{L^2(\Omega)}, \label{eq:stokes_zeroavg_0}\\  
a_{\int} ((\bm{u},p), (\bm{\tau}, q)) & = a((\bm{u},p), (\bm{\tau}, q)) +  (\langle p \rangle, \langle q \rangle)_{\Omega}.  \label{eq:stokes_zeroavg_1}
\end{align}
An immediate application of Lemma \ref{lemma:coercivity_stokes} and the definition of $a_{\int}$ now yields 
\begin{equation}
 \| \bm u \|^2_{H^{1/2}(\Omega)} 
    + 
    \|p\|^2_{(H^{1/2}(\Omega)\cap L^2_0(\Omega))^*}  + 
\|\langle p \rangle \|^2 _{L^2(\Omega)} \lesssim a_{\int}((\bm{u},p) , (\bm{u} , p)).
\end{equation}
Similar to the proof of bound \eqref{eq:error_estimate_stokes}, the above coercivity bound can be used to obtain additional error estimates for $\|\langle p_{\theta^*} \rangle \|_{L^2(\Omega)}^2$.
\end{remark}
\begin{remark}[Exact Boundary Values]
If we assume that our ansatz class satisfies the boundary conditions exactly, i.e., it holds $\mathcal F_{\Theta_1} \subset (H^2(\Omega)\cap H^1_{g_D}(\Omega))^d$, then the forms given by \eqref{eq:stokes_zeroavg_0} and \eqref{eq:stokes_zeroavg_1} are  modified to 
\begin{align*}
    E_{0}(\bm u, p) 
    &=
    E(\bm u, p) - \frac 12 \|\bm u - \bm g_D\|_{L^2(\partial \Omega)}^2, 
    \\   
    a_{0} ((\bm{u},p), (\bm{\tau}, q))  
    &=
    a((\bm{u},p), (\bm{\tau}, q)) - (\bm u, \bm v)_{\partial \Omega}. 
\end{align*}
From \eqref{eq:stokes_h1_regularity}, it is then evident that $a_0$ is coercive on $(H^2(\Omega) \cap H^1_0(\Omega))^d \times H^1(\Omega) \cap L^2_0(\Omega)$.  From similar arguments to Theorem~\ref{thm:stokes}, we then have the following estimate: 
\begin{equation}
    \label{eq:error_estimate_stokes_2}
    \|\bm u - \bm{u}_{\theta_1^*} \|^2_{H^{1}(\Omega)}
    +
    \|p - p_{\theta_2^*}\|^2_{L^2(\Omega)}   
    \lesssim L(\theta^*) + \eta(\theta^*). 
\end{equation} 
An a priori estimate in the above norms can be similarly derived. 
\end{remark}
\begin{remark}[Exact Boundary Values and Exact Divergence] If we further restrict our ansatz class to $\tilde{\mathcal{F}}_{\Theta_1}$ with exact boundary data and with $\nabla \cdot \bm w= 0$, $\forall \bm w  \in \tilde{\mathcal{F}}_{\Theta_1}$, then, from \eqref{eq:stokes_h2_regularity} and similar arguments to  Theorem~\ref{thm:stokes}, error estimates in the $H^2$--norm of the velocity and $H^1$--norm for the pressure can be established. We skip the details for brevity.   
\end{remark}
\subsection{Parabolic Equations}\label{sec:parabolic}
We consider the initial value problem
\begin{alignat}{2}
    \partial_tu - \nabla \cdot A \nabla u &= f,   &&\quad \text{in }I\times\Omega,
    \\
    u &= g_D,  && \quad \text{on }I\times\partial\Omega,
    \\
    u(0) &= u_0, && \quad \text{in }\Omega. 
\end{alignat}
We assume that $I 
\subset \mathbb R$,   $\Omega\subset\mathbb R^d$ is of $C^{1,1}$ regularity, the forcing $f$ is a member of $L^2(I,L^2(\Omega))$, the boundary data $g_D \in L^2(I, H^{3/2}(\partial\Omega))\cap H^{3/4}(I,L^2(\partial\Omega))$, $u_0 \in H^1(\Omega)$ with $g(0) = u_0 \vert_{\partial \Omega}$,  and $A\in C^{0,1}(\Omega, \mathbb R^{d\times d})$ is uniformly elliptic and symmetric. Under these assumptions the equation above has a unique solution of regularity $L^2(I,H^2(\Omega)) \cap H^1(I, L^2(\Omega))$, see \cite[Chapter 4, Theorem 6.1]{lions2012non} or \cite{weidemaier2002maximal}. The associated least squares energy is
\begin{align}\label{eq:E_parabolic}
    E(u) 
    &=
    \frac12\|\partial_t u - \nabla\cdot A\nabla u - f\|^2_{L^2(I,L^2(\Omega))}
    +
    \frac12\| u - g_D \|^2_{L^2(I,L^2(\partial\Omega))}
    \\
    &+
    \frac12 \|u(0) - u_0\|^2_{L^2(\Omega)}. \nonumber
\end{align}
The PINNs formulation is to inexactly solve
\begin{equation}
    \min_{\theta \in \Theta} L(\theta), \quad L(\theta) \approx E(u_{\theta}). 
\end{equation}
Note that the above is a space time formulation. The form $a$ is given by $a:[H^1(I,L^2(\Omega))\cap L^2(I,H^2(\Omega))]^2 \to \mathbb R$
\begin{align*}
    a(u,v)
    &=
    (\partial_tu - \nabla\cdot A\nabla u, \partial_tv - \nabla\cdot A\nabla v)_{L^2(I,L^2(\Omega))}
    \\
    &+
    (u,v)_{L^2(I,L^2(\partial\Omega))}
    +
    (u(0),v(0))_{L^2(\Omega)}.
\end{align*}
We can derive the following coercivity estimate.
\begin{lemma}[Coercivity]\label{eq:coerc_global_parabolic} Assume we are in the setting above, then for any $u \in H^1(I,L^2(\Omega))\cap L^2(I,H^2(\Omega))$ we can estimate
\begin{align*}
    \|u\|^2_{L^2(I,H^{1/2}(\Omega))}
    &\lesssim a(u,u)
\end{align*}
\end{lemma}
\begin{remark}
    In fact, the coercivity holds in the slightly stronger norm of 
    \[
        [L^2(I,L^2(\Omega)), H^1(I,H^1_0(\Omega)^*)\cap L^2(I, H^1(\Omega))]_{1/2,2}.
    \] 
    We omitted this in the statement above for brevity.
\end{remark}
\begin{proof}
    We outline the proof and give details in the appendix. We set 
    \[
        W = H^1(I, L^2(\Omega))\cap L^2(I, H^2(\Omega)\cap H^1_0(\Omega)),
    \] 
    Maximal parabolic regularity for the \emph{adjoint final value} problem implies for $u\in W$ with vanishing final time 
    \begin{align}\label{eq:parabolic_max_regularity}
        \| u \|_{W} 
        \lesssim
        \| -\partial_t u - \nabla\cdot A^T\nabla u \|_{L^2(I,L^2(\Omega))}.
    \end{align}
    Considering a dual problem, see Section~\ref{sec:dual_parabolic} for details,  we estimate for all $u\in H^1(I, L^2(\Omega))\cap L^2(I,H^2(\Omega))$
    \begin{align}\label{eq:parabolic_dual_estimate}
    \begin{split}
        \| u \|_{L^2(I,L^2(\Omega))}
        &
        \lesssim
        \|u(0)\|_{L^2(\Omega)} 
        +
        \| \partial_t u - \nabla \cdot A \nabla u \|_{W^*}
        \\
        &+
        \| u \|_{{(L^2(I, H^{1/2}(\partial\Omega))\cap H^{1/4}(I, L^2(\partial\Omega)))^*}}.
        \end{split}
    \end{align}
    For a reference concerning the regularity on the parabolic boundary we refer to \cite{weidemaier2002maximal}. The 
 energy estimates, see 
    \cite[Section 15.5, page 83]{lions2012non}, show that for $u\in H^1(I, L^2(\Omega))\cap L^2(I,H^2(\Omega))$
\begin{align}\label{eq:parabolic_energy_estimate}
        \begin{split}
            \| u \|_{H^1(I,H^1_0(\Omega)^*)\cap L^2(I, H^1(\Omega))}
            &
            \lesssim
            \|u(0)\|_{L^2(\Omega)}
            \\
            &+
            \| \partial_t u - \nabla\cdot A\nabla u \|_{L^2(I, H^1_0(\Omega))^*}
            \\
            &+
            \|u\|_{L^2(I, H^{1/2}(\partial\Omega))\cap H^{1/4}(I, L^2(\partial\Omega))},
        \end{split}
    \end{align}
    where the differential operator is understood in the weak form. Interpolating \eqref{eq:parabolic_dual_estimate} and \eqref{eq:parabolic_energy_estimate} and using a similar result to \eqref{eq:interpolation_dual} (see \cite{cobos1998theorem} for the abstract statement) for the boundary term yields 
    \begin{align*}
        \|u\|_Z
        &\lesssim
        \|u(0)\|_{L^2(\Omega)}
        +
        \|\partial_t u -\nabla\cdot A\nabla u\|_X
        \\
        &+
        \|u\|_{L^2(I, L^2(\partial \Omega))} ,
    \end{align*}
    where $X$ and $Z$ denote the interpolation spaces
    \begin{align*}
        X
        &=
        [W^*, L^2(I, H^1_0(\Omega)^*)]_{1/2,2}
        \\
        Z
        &=
        [L^2(I,L^2(\Omega)), H^1(I,H^1_0(\Omega)^*)\cap L^2(I, H^1(\Omega))]_{1/2,2}
    \end{align*}
    We can trivially estimate $\|\partial_t u -\nabla\cdot A\nabla u\|_X \lesssim \|\partial_t u -\nabla\cdot A\nabla u\|_{L^2(I,L^2(\Omega))}$. From \eqref{eq:interpolation_bochner}, we also have that $\|u\|_{L^2(I, H^{1/2}(\Omega))} \lesssim \| u \|_{Z}$. This  concludes the proof.
\end{proof}
 
\begin{theorem}[Error Estimate]
   Let $\theta^*$ inexactly minimize $L(\theta) \approx E(u_{\theta})$ where $E$ is given in \eqref{eq:E_parabolic}. Let  $u_{\theta^*}$ be the corresponding neural network. Then, 
   \begin{equation}
      \|u_{\theta^*} - u\|^2_{L^2(I, H^{1/2}(\Omega))} 
    \lesssim L(\theta^*) + \eta(\theta^*),  \label{eq:estimate_parabolic_aposter}
   \end{equation}
   where $\eta(\theta^*) = E(u_{\theta^*}) - L(\theta^*).$ Further, we have that 
\begin{equation}
    \|u_{\theta^*} - u\|^2_{L^2(I, H^{1/2}(\Omega))} 
    \lesssim
    \delta(u_{\theta^*}) 
    +
    \inf_{u_\theta \in \mathcal{F}_\Theta} \|u_{\theta} - u^*\|^2_{H^1(I, L^2(\Omega))\cap L^2(I, H^2(\Omega))},\label{eq:estimate_parabolic}
\end{equation}
where $\delta(u_{\theta^*})$ is given in \eqref{eq:delta_opt}. 
\end{theorem}
\begin{proof}
    The results follows from applications of Theorems \ref{thm:aposteriori_thm} and \ref{thm:cea}. We verify the  continuity of $a$ on $X=H^1(I,L^2(\Omega))\cap L^2(I,H^2(\Omega))$. Indeed, we have that 
    \begin{equation}
        a(u,v) 
        \leq
        \max(\|A\|^2_{C^{0,1}}, C_{\mathrm{tr}_{I\times\partial\Omega}}^2, C_{\mathrm{tr}_{t=0}}^2)
        \|u\|_{X} \|v\|_{X} 
    \end{equation}
    with $\|u\|_X = \|u\|_{H^1(I,L^2(\Omega))\cap L^2(I,H^2(\Omega))}$.
    To show \eqref{eq:estimate_parabolic} and \eqref{eq:estimate_parabolic_aposter}, we apply Theorems \ref{thm:cea} and \ref{thm:aposteriori_thm} respectively with $|||\cdot |||_{X} = \|\cdot\|_{L^2(I, H^{1/2}(\Omega))}$ where  the  coercivity estimate from Lemma \eqref{eq:coerc_global_parabolic} is used.   
\end{proof}

\subsection{Second-Order Hyperbolic Equations} We consider the initial value problem
\begin{alignat}{2}
    \partial_{tt} u - \nabla \cdot A \nabla u &= f &&  \quad \text{in }I\times\Omega
    \\
    u &= g_D && \quad \text{on }I\times\partial\Omega
    \\
    u(0) &= u_0,  \,\, 
    \partial_t u (0) = u_1 && \quad \text{in }\Omega,
\end{alignat}
where we assume that $\Omega\subset\mathbb R^d$ is of sufficient regularity, the forcing $f \in H^1(I,L^2(\Omega))$, the boundary data $g \in H^1(I, H^{3/2}(\partial\Omega))\cap H^{5/2}(I,L^2(\partial\Omega))$, $u_0 \in H^2(\Omega)$ and $u_1
\in H^{3/2}(\Omega)$ with $g(0) = u_0$ and $\partial_t g(0) = u_1$,  and $A\in C^{0,1}(\Omega, \mathbb R^{d\times d})$ is uniformly elliptic and symmetric. Under these assumptions the equation above has a unique solution of regularity $L^2(I,H^2(\Omega)) \cap H^2(I, L^2(\Omega))$, see \cite[Chapter 5, Theorem 3.1]{lions2012non}. Then, we can properly define the least squares energy: 
\begin{align}
    E(u) \label{eq:E_hyperbolic}
    &=
    \frac12\|\partial_{tt} u - \nabla\cdot A\nabla u - f\|^2_{L^2(I,L^2(\Omega))}
    +
    \frac12\| u - g_D \|^2_{L^2(I,L^2(\partial\Omega))}
    \\
    &+
    \frac12 \|u(0) - u_0\|^2_{L^2(\Omega)} + \frac12 \| \partial_t u(0) - u_1\|^2_{L^2(\Omega)} .  \nonumber
\end{align}
The form $a$ is given by $a:[ L^2(I,H^2(\Omega)) \cap H^2(I,L^2(\Omega))]^2 \to \mathbb R$
\begin{align}
    a(u,v) \label{eq:form_a_hyperbolic}
    &=
    (\partial_{tt} u - \nabla\cdot A\nabla u, \partial_{tt} v - \nabla\cdot A\nabla v)_{L^2(I,L^2(\Omega))}
    \\
    &+
    (u,v)_{L^2(I,L^2(\partial\Omega))}
    +
    (u(0),v(0))_{L^2(\Omega)} + (\partial_t u(0),\partial_t v(0))_{L^2(\Omega)}. \nonumber
\end{align}
Now, we derive the following coercivity estimate.
\begin{lemma}[Coercivity]\label{eq:coerc_global_hyperbolic} Assume we are in the setting above, then for any $u \in  L^2(I,H^2(\Omega)) 
\cap  H^2(I,L^2(\Omega))$ we can estimate
\begin{align}
\|u\|^2_{H_0^{3/4}(I,H^{1/4}(\Omega))^*} +  \|u\|^2_{H_{00}^{1/2}(I,L^2(\Omega))^*} 
    &\lesssim a(u,u) . \label{eq:coerc_hyperbolic}
\end{align}
Here,  $H_{00}^{1/2}(I,L^2(\Omega)) = [L^2(I,L^2(\Omega)),H^{1}_0(I,L^2(\Omega))]_{1/2,2}$ is the Lions--Magenes space. 
\end{lemma}
\begin{proof}
 Here too, we only outline the proof and give details in the appendix. We set 
     $H^{2,2}(I \times \Omega ) = L^2(I;H^2(\Omega)\cap H^1_0(\Omega)) \cap H^2(I;L^2(\Omega))$ and $H^{1/2,1/2}(I \times \partial \Omega ) =L^2(I, H^{1/2}(\partial \Omega)) \cap H^{1/2}(I,L^2(\partial \Omega))$. We can then show that, see Section~\ref{sec:dual_hyperbolic} for details, 
\begin{multline}
\|u\|_{H^1(I;L^2(\Omega))^*} \lesssim \|\partial_{tt} u - \nabla \cdot(A \nabla u )\|_{H^{2,2}(I \times \Omega)^*} + \|u\|_{H^{1/2,1/2}(\partial I \times \Omega)^*} \\  
+ \|u(0)\|_{H^{3/2}(\Omega)^*}+ \|\partial_t u(0)\|_{H^{1/2}(\Omega)^*}.  \label{eq:hyper_interp0}
\end{multline} 
From \cite[Chapter 3, Theorem 3.2]{lions2012non}, we also have that 
\begin{multline}
    \|u\|_{L^2(I;H^1(\Omega)) \cap H^1(I;L^2(\Omega))} 
    \lesssim
    \|\partial_{tt} u - \nabla \cdot(A \nabla u )\|_{L^2(I, L^2(\Omega))}  \\ + \|u\|_{L^2(I,H^{3/2}(\partial \Omega)) \cap H^{3/2}(I,L^2(\partial \Omega))}    
+ \|u(0)\|_{H^{3/2}(\Omega)}+ \|\partial_t u(0)\|_{H^{1/2}(\Omega)}.  \label{eq:hyper_interp1}
\end{multline}
We interpolate between the above bounds with $\theta = 1/4$. The result follows by invoking the interpolation result in \cite[Chapter 3, Theorem 5.4]{lions2012non} stating that 
\begin{equation*}
[H_0^1(I,L^2(\Omega))^*,L^2(I,H^1(\Omega)) \cap H^1(I,L^2(\Omega))]_{1/4,2}   = H_0^{3/4}(I,H^{1/4}(\Omega))^*\cap H_{00}^{1/2}(I,L^2(\Omega))^*.
\end{equation*}
\end{proof}

\begin{theorem}[Error Estimate]
    Let $\theta^*$ inexactly minimize $L(\theta) \approx E(u_{\theta})$ where $E$ is given in \eqref{eq:E_hyperbolic}. Let  $u_{\theta^*}$ be the corresponding neural network. 
    Then, 
    \begin{equation}
          \|u - u_{\theta^*}\|_{H_0^{3/4}(I,H^{1/4}(\Omega))^*}^2 +  \|u - u_{\theta^*}\|_{H_{00}^{1/2}(I,L^2(\Omega))^*}^2 
        \lesssim  L(\theta^*) + \eta(\theta^*),  \label{eq:estimate_hyperbolic_apost}
    \end{equation}
    where $\eta(\theta^*) = E(u_{\theta^*}) - L(\theta^*)$.
    \begin{multline}
        \|u - u_{\theta^*}\|_{H_0^{3/4}(I,H^{1/4}(\Omega))^*}^2 +  \|u - u_{\theta^*}\|_{H_{00}^{1/2}(I,L^2(\Omega))^*}^2 
        \lesssim 
        \delta(u_{\theta^*}) \\  + \inf_{u_\theta \in \mathcal{F}_\Theta} \|u_{\theta} - u^*\|^2_{ L^2(I, H^2(\Omega)) \cap H^2(I,L^2(\Omega))}, \label{eq:estimate_hyperbolic}
    \end{multline}
    with $\delta(u_{\theta^*})$ given in \eqref{eq:delta_opt}. 
    \end{theorem}
    \begin{proof}
      Estimate \eqref{eq:estimate_hyperbolic_apost} follows from applying Theorem \ref{thm:aposteriori_thm} with the coercivity estimate \eqref{eq:coerc_hyperbolic}.   Estimate \eqref{eq:estimate_hyperbolic} follows from applying Theorem \ref{thm:cea} once we verify the  continuity of $a$ on $X=L^2(I,H^2(\Omega))\cap H^2(I,L^2(\Omega))$. First note that  the trace operator maps continuously from the space $X$ to the space $L^2(I,L^2(\partial \Omega))$ and that 
        \[\|u(0)\|_{L^2(\Omega)}+  \|\partial_t u(0)\|_{L^2(\Omega)} \lesssim \|u\|_{X},  \quad \forall u \in X.\]
    For the above properties, we refer to \cite[Chapter 4, Theorem 2.1]{lions2012non}. Continuity readily follows. We skip the details for brevity. This property and \eqref{eq:coerc_hyperbolic} shows  \eqref{eq:estimate_hyperbolic}.  
 \end{proof}
\begin{remark}[Exact Boundary Values]
If we assume that our ansatz class satisfies the boundary values exactly, i.e., it holds $\mathcal F_{\Theta} \subset L^2(I,H^2(\Omega) \cap H^1_{g_D}(\Omega))$, then the forms given in \eqref{eq:E_hyperbolic} and \eqref{eq:form_a_hyperbolic} are modified to 
 \begin{align*}
     E_0(u) = E(u) - \frac12 \|u-g\|^2_{L^2(I,L^2(\partial \Omega))}, 
    \quad a_0(u,v) = a(u,v) - (u,v)_{L^2(I,L^2(\partial \Omega)}. 
 \end{align*}
 It then suffices to establish coercivity of $a$ over the space $L^2(I,H^2(\Omega) \cap H^1_0(\Omega)) \cap H^2(I,L^2(\Omega))$. To this end, observe that 
 the boundary terms in \eqref{eq:hyper_interp0} and in \eqref{eq:hyper_interp1} are  zero over this space. We can then interpolate with $\theta = 1/2$ and we obtain a coercivity estimate of $a_0$ in $L^2(I, L^2(\Omega))$.  In this case, one obtains an an error estimate in $L^2(I,L^2(\Omega))$: 
 \[
    \|u - u_{\theta^*}\|_{L^2(I,L^2(\Omega))}^2  
    \lesssim
    \delta(u_{\theta^*})  + \inf_{u_\theta \in \mathcal{F}_\Theta} \|u_{\theta} - u^*\|^2_{ L^2(I, H^2(\Omega)) \cap H^2(I,L^2(\Omega))}. 
\]
An a posteriori bound in the above norm also holds in this case. 
\end{remark}

\subsection{Source recovery problem for the Laplacian}\label{sec:inverse_problem}
Before we start considering inverse problems, we need to introduce a space tailored to the problem. This is due to the fact that the existence of minimizers of the problem formulation is not directly clear. Assume we are given a domain $\Omega\subset \mathbb R^d$ with $C^{1,1}$ boundary. Then,  we define the space
\begin{equation}\label{eq:definition_H}
    \mathcal H = \overline{C^\infty(\bar\Omega)}^{\|\cdot\|_{\mathcal H}}, \quad \text{with} \quad \|u\|_{\mathcal H} = \|\Delta u\|_{L^2(\Omega)} + \|u\|_{L^2(\partial\Omega)}.
\end{equation}
\begin{lemma}[Properties of $\mathcal H$]
    The space $\mathcal H$ is a Hilbert space and $\mathcal H\subset H^{1/2}(\Omega)$, meaning that elements of $\mathcal H$ can be identified with $H^{1/2}(\Omega)$ functions.
\end{lemma}
\begin{proof} The norm $\|\cdot\|_{\mathcal H}$ is induced by the bilinear form
    \begin{equation*}
        (u,v)_{\mathcal H} 
        =
        \int_\Omega \Delta u\Delta v \, \mathrm dx + \int_{\partial \Omega}uv \, \mathrm ds.
    \end{equation*}
    As we are in an $H^2$ regular setting,  we know from Lemma \ref{lemma:coercivity_poisson} that
    \begin{equation*}
        \|u\|_{H^{1/2}(\Omega)} \lesssim  \|u\|_{\mathcal H}, 
    \end{equation*}
    which shows that $(\cdot, \cdot)_{\mathcal H}$ is positive definite and hence an inner product. This implies that the abstract completion $\mathcal H$ is a Hilbert space. Further, any Cauchy sequence $(\varphi_n)\subset C^\infty(\bar\Omega)$ with respect to $\|\cdot\|_{\mathcal H}$ is bounded in $H^{1/2}(\Omega)$ and thus has a weakly $H^{1/2}(\Omega)$ convergent subsequence $\varphi_{k} \rightharpoonup u$ in $H^{1/2}(\Omega)$. The limit $u$ must be unique since assuming there are two different limits $u_1$ and $u_2$ for different subsequences leads to contradiction to the $\|\cdot\|_{\mathcal H}$ Cauchy property of $(\varphi_n)$. 
\end{proof}

\begin{remark}
    Note that passing to the completion $\mathcal H$ requires to extend the operator $\Delta$ and $\operatorname{tr}$ to the completion, too. We do this without introducing further notation, understanding that if these operators are applied to general $\mathcal H$ functions it is meant in the sense of a limit. Continuity and linearity are trivially inherited.
\end{remark}

\begin{lemma}
    Define the energy $E:\mathcal H \times L^2(\Omega) \to \mathbb R$ with
    \begin{equation} \label{eq:energy_inverse}
        E(u,f) 
        = 
        \frac12 \| \Delta u + f \|^2_{L^2(\Omega)}
        +
        \frac12 \| u - u_d \|^2_{L^2(\Omega)}
        +
        \frac12 \| u - g \|^2_{L^2(\partial\Omega)}
        +
        \frac{\eta^2}{2}\|f\|^2_{L^2(\Omega)}
    \end{equation}
    for interior observations $u_d \in L^2(\Omega)$, boundary values $g\in L^2(\partial\Omega)$ and a regularization parameter $\eta >0$. Then $E$ has a unique minimizer.
\end{lemma}
\begin{proof}
    Let $(u_n), (f_n)$ denote a minimizing sequence, that means $E(u_n, f_n) \to \inf_{u,f}E(u,f)$. We now need to bound $(u_n), (f_n)$ with respect to $\|\cdot\|_{\mathcal H}$ and $\|\cdot\|_{L^2(\Omega)}$. This can be done crucially relying on the fact that $\eta > 0$ by standard arguments involving Young's inequality. One arrives at a bound of the form
    \begin{equation*}
        E(u, f) 
        \geq 
        \varepsilon \|\Delta u\|^2_{L^2(\Omega)} + \varepsilon \|f\|_{L^2(\Omega)}^2 + \tilde \varepsilon \|u\|^2_{L^2(\partial\Omega)} - c,
    \end{equation*}
    for suitable $\varepsilon > 0$ depending on $\eta$ and $\tilde\varepsilon>0$ depending on $\|g\|_{L^2(\partial\Omega)}$. This gives a bound of the form
    \begin{equation*}
        \|\Delta u_n\|^2_{L^2(\Omega)} + \|f_n\|^2_{L^2(\Omega)} + \|u_n\|^2_{L^2(\partial\Omega)} < \tilde c.
    \end{equation*}
    Hence there is a subsequence $u_k\rightharpoonup u^*$ in $\mathcal H$ and $f_k\rightharpoonup f^*$ in $L^2(\Omega)$. We then note that $E$ is convex and norm continuous, hence weakly lower semi-continuous. This yields
    \begin{equation*}
        \inf_{u\in\mathcal H, f\in L^2(\Omega)}E(u,f) = \liminf_{n\to\infty}E(u_n, f_n) \geq E(u^*, f^*) \geq \inf_{u\in\mathcal H, f\in L^2(\Omega)}E(u,f).
    \end{equation*}
    This means that $(u^*,f^*)$ is a minimizer and the fact that $E$ is strictly convex shows that it is the only minimizer.
\end{proof}

Consider the problem of minimizing
\begin{equation*}
    \min_{u\in\mathcal H, f\in L^2(\Omega)} J(u,f) = \frac12 \|u - u_d\|_{L^2(\Omega)}^2 + \frac{\eta^2}{2}\|f\|^2_{L^2(\Omega)}
\end{equation*}
subjected to the PDE constraint
\begin{align*}
    -\Delta u &=f \quad \text{in }\Omega
    \\
    u &=g \quad \text{on }\partial\Omega,
\end{align*}
where $u_d\in L^2(\Omega)$ is a given interior observation and $\eta > 0$ is a penalization parameter. The least squares formulation yields $E:\mathcal H \times L^2(\Omega) \to \mathbb R$
\begin{equation}
    E(u,f) 
    = 
    \frac12 \| \Delta u + f \|^2_{L^2(\Omega)}
    +
    \frac12 \| u - u_d \|^2_{L^2(\Omega)}
    +
    \frac12 \| u - g \|^2_{L^2(\partial\Omega)}
    +
    \frac{\eta^2}{2}\|f\|^2_{L^2(\Omega)}. 
\end{equation}
The discussion above guarantees now that $E$ possesses a unique minimizer in $\mathcal H \times L^2(\Omega)$. The associated form is
\begin{equation*}
    a((u,f), (v, h)) 
    =
    (\Delta u + f, \Delta v + h )_{\Omega} 
    +
    (u,v)_{\partial \Omega} 
    +
    (u,v)_{\Omega} 
    +
    \eta^2 (f,h)_{\Omega}. 
\end{equation*}

\begin{lemma}[Coercivity]  \label{lemma:coercivity_source_recov}
    For a given $\beta \in (0,1)$, assume that $\eta^2 > \beta/(1-\beta)$. There holds 
    \begin{equation}
        a((u,f),(u,f))  
        \gtrsim
        C_1 \beta \|u\|_{H^{1/2}(\Omega)}^2 
        +
        \|u\|^2_{L^2(\Omega)} 
        +
        (1-\beta) \|u\|^2_{L^2(\partial \Omega)} 
        +
        C_3 \|f\|^2_{L^2(\Omega)}, \label{eq:second_coerc_inverse}
    \end{equation}
    with $C_3 = \eta^2 - \beta/(1 - \beta)$.
\end{lemma}

\begin{proof}
    We have that  
    \begin{align*}
        a((u,f),(u,f)) - 2(f,\Delta u) 
        \geq
        \| \Delta u\|_{L^2(\Omega)}^2 
        +
        \|u\|_{L^2(\Omega)}^2  
        +
        \|u\|^2_{L^2(\partial \Omega)}
        +
        (\eta^2+1) \|f\|_{L^2(\Omega)}^2. 
    \end{align*}
    With Cauchy--Schwarz and Young's inequalities, we estimate for $\epsilon > 0$:
    \begin{align*}
        2|(f,\Delta u)| 
        \leq
        \epsilon \|\Delta u \|_{L^2(\Omega)}^2 
        +
        \frac{1}{\epsilon} \|f\|^2_{L^2(\Omega)}  
    \end{align*}
    It then follows that 
    \begin{multline}
        a((u,f),(u,f)) 
        \geq
        (1-\epsilon) \left( \|\Delta u \|_{L^2(\Omega)}^2 
        +
        \|u\|^2_{L^2(\partial \Omega)}  \right) 
        \\ + 
        \|u\|^2_{L^2(\Omega)} 
        +
        \epsilon \|u\|_{L^2(\partial \Omega) }^2 
        +
        (\eta^2 - \frac{1}{\epsilon} + 1) \|f\|_{L^2(\Omega)}^2.
    \end{multline}
    Now, choose  $\epsilon = 1-\beta $.  Using \eqref{eq:coerc_global_poisson} gives \eqref{eq:second_coerc_inverse}.
\end{proof}

The above results directly yield error estimates.
\begin{theorem}[Error Estimates] 
Let  $(u, f ) \in \mathcal H \times L^2(\Omega)$ be the exact solution and let $\theta^* = (\theta_1^*, \theta_2^*) $ inexactly minimize    $L( \theta) = L(\theta_1, \theta_2) \approx E(u_{\theta_1}, f_{\theta_2})$ where $E$ is given in \eqref{eq:energy_inverse}. Let  $\bm u_{ \theta_1^*}$ and $f_{\theta_2^*}$ be the corresponding neural networks. Let  $\eta$, $\beta$, $C_1$ and $C_3$ be as in Lemma \ref{lemma:coercivity_source_recov}. Then, 
\begin{equation} \label{eq:aposter_source_recov}
     C_1 \beta \|u_{\theta_1^*} - u  \|_{H^{1/2}(\Omega)}^2  
        +
     (1-\beta)   \|u_{\theta_1^*} - u \|_{L^2(\partial \Omega)}^2 
        +
       C_3  \|f_{\theta_2^*} - f \|_{L^2(\Omega)}^2 
        \lesssim L(\theta^*) + \eta(\theta^*), 
\end{equation}
where $\eta(\theta^*) = E((u_{\theta_1^*}, f_{\theta_2^*})) - L(\theta_1^*, \theta_2^*)$. Further, we have that  
    \begin{multline} \label{eq:aprior_source_recov}
        C_1 \beta \|u_{\theta_1^*} - u  \|_{H^{1/2}(\Omega)}^2  
        +
     (1-\beta)   \|u_{\theta_1^*} - u \|_{L^2(\partial \Omega)}^2 
        +
       C_3  \|f_{\theta_2^*} - f \|_{L^2(\Omega)}^2 
        \lesssim
        \delta(( u_{\theta_1^*}, f_{\theta_2^*}))  
        \\ +
        \inf_{(u_{\theta_1}, f_{\theta_2}) \in \mathcal{F}_{\Theta_1} \times \mathcal{F}_{\Theta_2}}
 ( \|u_{\theta_1} - u \|_{\mathcal H}^2 
        +
        \| f_{\theta_2} - f \|^2_{L^2(\Omega)}),   
    \end{multline}
    where $\delta(( u_{\theta_1^*}, f_{\theta_2^*})) =E(u_{\theta_1^*}, f_{\theta_2^*}) - \inf_{(\bm u_{\psi_1}, f_{\psi_2} ) \in \mathcal{F}_{\Theta_1} \times \mathcal{F}_{\Theta_2} } E(u_{\psi_1} , f_{\psi_2} ) $
\end{theorem}
\begin{proof}
For \eqref{eq:aposter_source_recov}, we apply Theorem \ref{thm:aposteriori_thm} where the coercivity estimate \eqref{eq:second_coerc_inverse} is used. Further, it can easily be seen that $a$ is continuous with respect to the norm of $\mathcal H \times L^2(\Omega)$ with constant $\max(1+\eta^2, C_{\mathcal H})^2$ where $C_{\mathcal H}$ is the continuity constant of the embedding $\mathcal H \hookrightarrow L^2(\Omega)$. Applying Theorem \ref{thm:cea} yields \eqref{eq:aprior_source_recov}. 
\end{proof}
\section{Numerical Experiments}\label{sec:numerics}
In this Section we present three or four dimensional numerical illustrations of the previously analyzed equations. 
Utilizing recent advances in optimization schemes for PINNs \cite{muller2023achieving}, we are able to efficiently and accurately solve the non-convex optimization problems resulting from the PINN formulation, in contrast to the notorious difficulties reported in the literature \cite{wang2021understanding, krishnapriyan2021characterizing, wang2022and}. In our experiments, computation times for a single PDE solve lie in the range of 1 - 5 minutes on a standard GPU-equipped Laptop.

\subsection{Experiment Details}
We numerically illustrate the previously analyzed equations using the PINN formulation for a manufactured solution setting, see Table~\ref{table:manufactured_soltuions} for the solutions we are using and Table \ref{table:errors_1} and \ref{table:errors_2} for the achieved errors. We use shallow neural networks with the $\tanh$ activation function as an ansatz. The optimization is carried out using the natural gradient method proposed in \cite{muller2023achieving} which we complement with a Levenberg-Marquardt type regularization, meaning the update direction $d_{k+1}$ is
\[
    d_{k+1} = [G(\theta_k) + \mu_k \operatorname{Id}]^{-1}\nabla L(\theta_k),
\]
and the update is computed as
\[
    \theta_{k+1} = \theta_k - \alpha_k d_{k+1}
\]
with a suitable chosen step-size $\alpha_k > 0$. By $G(\theta_k)$ we denote the Gram (or Fisher) matrix, see \cite{muller2023achieving}, $\mu_k > 0$ is a damping parameter and $\nabla L(\theta_k)$ denotes the gradient of the loss for neural network parameters $\theta$. For all our experiments we set 
\[
    \mu_k = \min(L(\theta_k), 1e-05)
\]
and optimize for 500 steps. The integral discretization is carried out using roughly 1000 uniformly drawn collocation points on the interior of $\Omega$ or $I\times\Omega$ and roughly 100 uniformly drawn points on the (parabolic) boundary $\partial\Omega$ or $I\times\partial\Omega \cup \{0\}\times\Omega$. These points are drawn once and kept fixed during the optimization process. The errors are evaluated at 10000 randomly drawn points that do not coincide with the training points.

\subsection{Implementation Details}
For the implementation we rely on the JAX framework \cite{jax2018github}. The experiemtns were run on a NVIDIA RTX 3080 Laptop GPU in double precision. Once the preprint is accepted for publication, the code to reproduce the experiments will be made publically available on Github.

\subsection{Discussion}
As documented in Tables \ref{table:errors_1} and \ref{table:errors_2} we achieve highly accurate solutions with moderate computing times of 1 to 5 minutes per PDE solve, depending on the equation. The employed networks are small and simple -- which we regard an advantage of the neural network ansatz class. In the appendix, we provide further experiments with larger networks and computational budget and observe significantly improved errors, compare to Table \ref{table:errors_accurate_1} and \ref{table:errors_accurate_2}.

Common optimization algorithms in the PINN literature are Adam or BFGS. As is well documented in the literature, these optimizers struggle to produce highly accurate solutions and typically need a large number of iterations and computation time, even for mediocre errors \cite{wang2021understanding, krishnapriyan2021characterizing, wang2022and}. Furthermore, laborious hyperparameter tuning, such as carefully weighting different contributions in the loss function \cite{wang2021understanding}, adaptive sampling schemes \cite{wu2023comprehensive} and learning rate schedules can be required for successful training. Our numerical experiments on the other hand, did not require any hyperparameter tuning. Instead, we are using the exact same network architecture, integral discretization and optimization algorithm for all considered equations. Crucial for the success of the numerical results is the recently proposed natural gradient method \cite{muller2023achieving}.     

\begin{table}[h!]
\begin{tabular}{|c|c|c|c|} \hline 
    PDE        &  $w$  & $L^2$ error & $H^1$ error       \\  \hline  \hline 
    Poisson    &  64   & $7.02e-05$  & $5.78e-04$        \\  \hline 
    Elasticity &  64   & $7.34e-04$  & $1.94e-02$        \\  \hline 
    Parabolic  &  64   & $1.82e-03$  & $1.54e-02$        \\  \hline 
    Hyperbolic &  64   & $4.41e-04$  & $3.81e-03$        \\  \hline
\end{tabular}    
\caption{Computed $L^2$ and $H^1$ errors of the PINN formulation using a $\tanh$ shallow neural network of width $w$. The manufactured solutions considered are given in Table~\ref{table:manufactured_soltuions} }
\label{table:errors_1}
\end{table}

\begin{table}[h!]
\begin{tabular}{|c|c|c|c|c|} \hline 
    PDE                & $w_1$ & $w_2$    & $L^2$  & $H^1$    \\  \hline  \hline 
 \multirow{2}{*}{Stokes}  & \multirow{2}{*}{32}    &  \multirow{2}{*}{32}        &  \multirow{2}{*}{$\bm u:\,\, 1.84e-04$}   & $\bm u:\,\, 2.91e-03$    \\   
                  &     &        &    & $p: \,\, 2.20e-02$    \\  \hline 
    \multirow{2}{*}{Transient Stokes}  & \multirow{2}{*}{32}    &  \multirow{2}{*}{32}        &  \multirow{2}{*}{$\bm u:\,\,  1.41e-04$}   & $\bm u:\,\, 2.30e-03$    \\   
                  &     &        &    & $p: \,\, 1.59e-02$    \\  \hline 
 \multirow{2}{*}{Darcy}  & \multirow{2}{*}{32}    &  \multirow{2}{*}{32}        &  $\bm \sigma:\,\, 1.34e-03 $  & \multirow{2}{*}{}\\   
 &     &        &  $p:\,\, 1.13e-04 $  &     \\  \hline    
\multirow{2}{*}{Inverse}  & \multirow{2}{*}{32}    &  \multirow{2}{*}{32}        &  $\bm u:\,\, 3.16e-03 $  & \multirow{2}{*}{$\bm u: \,\, 4.73e-02 $ }  \\   
                  &     &        &  $f: \,\, 5.66e-02 $  &     \\  \hline    \end{tabular}    
\caption{Computed $L^2$ and $H^1$ errors of the PINN formulation using a $\tanh$ shallow neural network of width $w$. The manufactured solutions considered are given in Table~\ref{table:manufactured_soltuions}. The  error in the pressure for the steady and transient Stokes equations is measured in the $H^1(\Omega)$ and $L^2(I,H^1(\Omega))$ semi-norms, respectively.}
\label{table:errors_2}
\end{table}

\begin{table}[h!]
\setlength{\tabcolsep}{10pt} 
\renewcommand{\arraystretch}{1.5} 
\begin{tabular}{|c|c|} \hline 
     & Manufactured solutions \\  \hline  \hline 
    Poisson  &  $\sin(\pi x)  \sin(\pi y) \sin(\pi z) $ \\  \hline 
   Elasticity & $(x^2 + 1)(y^2+1)(z^2+1)e^{x+y+z} \bm 1 $  \\  \hline 
Darcy & $ \bm \sigma = \nabla p  $, $p=\sin(\pi x)  \sin(\pi y) \sin(\pi z)  $ \\ \hline
    Parabolic  &$e^{-\pi^2 t/4} (\cos(\pi x)+ \cos(\pi y) + \cos(\pi z)) $ \\  \hline 
    Hyperbolic  &$\sin(\pi t) \sin(\pi x)  \sin(\pi y) \sin(\pi z) $ \\ \hline
\multirow{2}{*}{Stokes} & $\bm u = \nabla \times (x^2y^2(x-1)^2(y-1)^2 , 0 , 0) $ \\   
 &  $p = xyz(1-x)(1-y)(1-z)$ \\ \hline 
Transient Stokes & $e^{-t/2} \bm u, \,\, e^{-t/2} p $\\  
\hline 
\end{tabular}    
\caption{Manufactured solutions for the various equations. For Poisson, Darcy, parabolic, and hyperbolic equations, we let $A = I$. For Elasticity, $\mathbb C (\epsilon) = \lambda \mathrm{tr}(\epsilon) I  + 2 \mu \epsilon$, $\lambda = 0.5769$, $\mu = 0.3846$.   }  
\label{table:manufactured_soltuions}
\end{table}
\section{Conclusions} We propose an abstract framework for the error analysis of PINNs for linear PDEs. Our results show that every neural network that sufficiently minimizes the PINN loss, is a good approximation to the PDEs solution. This is a stronger statement than proving the mere existence of a neural network that is capable of approximating the solution. Our framework relies on the coercivity and continuity properties of a suitably defined bilinear form and leads to a priori and  a posteriori estimates of PINN formulations.  We utilize this framework to analyze the convergence of PINNs for elliptic PDEs (Poisson and Elasticity), Darcy's equation, a source recovery problem, Stokes system, parabolic, and hyperbolic equations. All the proofs follow a common structure based on energy estimates, dual estimates derived from regularity properties, and interpolation. Numerical tests for all the analyzed PDEs show that PINNs, with new optimization algorithms, can provide accurate solutions efficiently. Future work can concentrate on extending the framework to nonlinear problems.  

\section{Acknowledgments}  We gratefully acknowledge valuable discussions with Prof.\,Siddhartha Mishra. 
\section{Appendix} Here, we present proofs for the dual estimates that we used throughout the paper.
\subsection{Dual  Estimates for Elliptic Equations}\label{sec:poisson_negative}

\begin{lemma}[Dual  Estimate for Laplacian]
Let $\Omega \subset \mathbb R^d$ be $H^2(\Omega)$ regular\footnote{This means that $-\Delta:H^2(\Omega)\cap H^1_0(\Omega) \to L^2(\Omega)$ is a linear homeomorphism.}. Then, for all $u\in H^2(\Omega)$ we may estimate 
\begin{equation*}
    \| u \|_{L^2(\Omega)} 
    \lesssim
     \| \Delta u \|_{(H^2(\Omega)\cap H^1_0(\Omega))^*} + \| u \|_{H^{1/2}(\partial\Omega)^*}.
\end{equation*}

\end{lemma}
\begin{proof}
    By assumption we have that 
    \begin{equation*}
        -\Delta : H^2(\Omega)\cap H^1_0(\Omega) \to L^2(\Omega), \quad u\mapsto -\Delta u
    \end{equation*}
    is a linear homeomorphism. Specifically this implies
    \begin{equation}\label{eq:reg_estimate_local_ref}
        \|u\|_{H^2(\Omega)} 
        \leq 
        C_{\textrm{reg}}\cdot \|\Delta u \|_{L^2(\Omega)},
    \end{equation}
    where $C_{\textrm{reg}} = \|(-\Delta)^{-1}\|_{\mathcal L(L^2(\Omega), H^2(\Omega)\cap H^1_0(\Omega))}$. Then also the map
    \begin{equation*}
        T:L^2(\Omega) \overset{R}{\longrightarrow} L^2(\Omega)^* \overset{(-\Delta)^{*}}{\longrightarrow} (H^2(\Omega)\cap H_0^1(\Omega))^* 
    \end{equation*}
    is a linear homeomorphism, where we compose the Riesz isomorphism $R$ of $L^2(\Omega)$ and the Banach space adjoint of $-\Delta$. This means we can estimate
    \begin{equation}\label{eq:reg_estimate_local_ref_2}
        \| u \|_{L^2(\Omega)} 
        \leq 
        C_{\textrm{reg}}\cdot \|Tu\|_{(H^2(\Omega)\cap H_0^1(\Omega))^*}.
    \end{equation}
    Note that $C_{\textrm{reg}}$ in \eqref{eq:reg_estimate_local_ref} and \eqref{eq:reg_estimate_local_ref_2} are the same constant, which is due to the isometry $T\mapsto T^*$. Now we assume that $u\in H^2(\Omega)$ and compute for $\varphi\in H^2(\Omega)\cap H^1_0(\Omega)$
    \begin{align*}
        \langle Tu, \varphi \rangle
        =
        -\int_\Omega u \Delta \varphi \ \mathrm dx 
        &=
        \int_\Omega \nabla u \nabla \varphi \ \mathrm dx - \int_{\partial\Omega} u \partial_n \varphi \ \mathrm ds
        \\ 
        &=
        -\int_\Omega \Delta u \varphi \ \mathrm dx - \int_{\partial\Omega} u \partial_n \varphi \ \mathrm ds
    \end{align*}
    Hence, we can estimate, using that $\partial_n\varphi \in H^{1/2}(\partial\Omega)$
    \begin{equation*}
        \| Tu \|_{(H^2(\Omega)\cap H_0^1(\Omega))^*}
        \leq 
        \| \Delta u \|_{(H^2(\Omega)\cap H_0^1(\Omega))^*} + \|u\|_{H^{1/2}(\partial\Omega)^*}.
    \end{equation*}
    Combining with \eqref{eq:reg_estimate_local_ref_2} concludes the proof.
\end{proof}

\begin{lemma}[Dual  Estimate for Elasticity]
Let $\Omega \subset \mathbb R^d$ be $H^2(\Omega)^d$ regular and assume the same data regularity as in Section~\ref{sec:elasticity}. Then, for all $\bm u\in H^2(\Omega)^d$ we may estimate 
\begin{equation*}
    \| \bm u \|_{L^2(\Omega)} 
    \lesssim 
    \| \nabla\cdot \mathbb C \varepsilon(\bm u) \|_{(H^2(\Omega)\cap H^1_0(\Omega))^*} + \| \bm u \|_{H^{1/2}(\partial\Omega)^*}. 
\end{equation*}
\end{lemma}
\begin{proof}
    By assumption we have that 
    \begin{equation*}
        -\nabla\cdot \mathbb C \varepsilon : H^2(\Omega)\cap H^1_0(\Omega) \to L^2(\Omega), \quad \bm u\mapsto -\nabla\cdot \mathbb C \varepsilon(\bm u)
    \end{equation*}
    is a linear homeomorphism. Specifically this implies
    \begin{equation}
        \|\bm u\|_{H^2(\Omega)} 
        \leq 
        C_{\textrm{reg}}\cdot \|\nabla\cdot \mathbb C \varepsilon(\bm u) \|_{L^2(\Omega)},
    \end{equation}
    Then also the map
    \begin{equation*}
        T:L^2(\Omega)^d \overset{R}{\longrightarrow} (L^2(\Omega)^d)^* \overset{(-\nabla\cdot \mathbb C \varepsilon)^{*}}{\longrightarrow} (H^2(\Omega)^d\cap H_0^1(\Omega)^d)^* 
    \end{equation*}
    is a linear homeomorphism, where we compose the Riesz isomorphism $R$ of $L^2(\Omega)^d$ and the Banach space adjoint of $-\nabla\cdot \mathbb C \varepsilon$. This means we can estimate
    \begin{equation}
        \| \bm u \|_{L^2(\Omega)} 
        \leq 
        C_{\textrm{reg}}\cdot \|T \bm u\|_{(H^2(\Omega)\cap H_0^1(\Omega))^*}.
    \end{equation}
    Now we assume that $u\in H^2(\Omega)^d$ and compute for $\varphi\in H^2(\Omega)^d\cap H^1_0(\Omega)^d$
    \begin{align*}
        \langle T \bm u, \bm \varphi \rangle
        =
        -\int_\Omega \bm u \cdot \nabla\cdot \mathbb C \varepsilon(\bm \varphi) \ \mathrm dx 
        &=
        \int_\Omega \nabla \bm u : \mathbb C \varepsilon(\bm \varphi) \ \mathrm dx - \int_{\partial\Omega} \bm u \partial_{\mathbb C} \bm \varphi \ \mathrm ds
        \\ 
        &=
        \int_\Omega  \varepsilon(\bm u) : \mathbb C \varepsilon(\bm \varphi) \ \mathrm dx - \int_{\partial\Omega} \bm  u \partial_{\mathbb C} \bm \varphi \ \mathrm ds
        \\
        &=
        -\int_\Omega \nabla\cdot\mathbb C\varepsilon(\bm u) \cdot \bm \varphi \ \mathrm dx - \int_{\partial\Omega} \bm u \partial_{\mathbb C} \bm \varphi \ \mathrm ds.
    \end{align*}
    setting $\partial_{\mathbb C} \bm \varphi = \mathbb C \varepsilon(\bm \varphi)\cdot \bm n$. Using the fact that $\partial_{\mathbb C} \bm \varphi \in H^{1/2}(\partial\Omega)^d$ we get
    \begin{equation*}
        \| T \bm u \|_{(H^2(\Omega)\cap H^1_0(\Omega))^*}
        \lesssim
        \|\nabla \cdot \mathbb C \varepsilon(\bm u)\|_{(H^2(\Omega)\cap H^1_0(\Omega))^*}
        +
        \| \bm u \|_{H^{1/2}(\partial\Omega)^*}
    \end{equation*}
    which concludes the proof.
\end{proof}

\subsection{Dual Estimate for Parabolic Equations}\label{sec:dual_parabolic}
\begin{lemma}[Dual Estimate] Assume we are in the setting of Section \ref{sec:parabolic} and set \[W = H^1(I, L^2(\Omega))\cap L^2(I, H^2(\Omega)\cap H^1_0(\Omega)),\] Then for all $u\in H^1(I,L^2(\Omega))\cap L^2(I, H^2(\Omega))$ we may estimate
\begin{align*}
    \|u\|_{L^2(I,L^2(\Omega))}
    &\lesssim
    \|u(0)\|_{L^2(\Omega)} 
    +
    \|\partial_tu - \nabla\cdot A\nabla u \|_{W^*}
    \\
    &+
    \| u \|_{{L^2(I, H^{1/2}(\partial\Omega))\cap H^{1/4}(I, L^2(\partial\Omega))^*}}.
\end{align*}
\end{lemma}
\begin{proof}
We write: 
\begin{align}
\|u\|_{L^2(I, L^2(\Omega))} = \sup_{\phi \in L^2(I;L^2(\Omega))} \frac{\int_I \int_{\Omega} u \phi }{\|\phi\|_{L^2(I;L^2(\Omega))}}. 
\end{align}
For a given $\phi \in L^2(I, L^2(\Omega))$, define the backward in time problem 
\begin{align}
-\partial_{t} \psi - \nabla \cdot (A^T \nabla  \psi)  = \phi ,  \quad \mathrm{in} \,\,  \Omega 
\end{align}
with zero boundary  and zero final time conditions. Then, by maximal parabolic regularity, there holds  
\begin{equation}
    \|\psi\|_{L^2(I, H^2(\Omega))\cap H^1(I,L^2(\Omega))}  \lesssim \|\phi\|_{L^2(I, L^2(\Omega))}. \label{eq:regular_parabolic_max}
\end{equation}
We now compute 
\begin{align*}
   \int_{I} \int_{\Omega} u \phi  
    &=
    -\int_I\int_\Omega u \partial_t \psi \, \mathrm dx\mathrm dt
    -
    \int_I\int_\Omega \nabla\cdot (A^T \nabla\psi) u  \, \mathrm dx\mathrm dt
    \\
    &=
    \int_\Omega u(0)\psi(0)\, \mathrm dx
    +
    \int_I\int_\Omega \partial_t u \psi \, \mathrm dx\mathrm dt
    \\
    &-
    \int_I\int_{\partial\Omega}u\partial^{A^T}_n\psi \,\mathrm dx\mathrm dt
    -
    \int_I\int_{\Omega}\nabla\cdot (A\nabla u)\psi \,\mathrm dx \mathrm dt,
\end{align*}
where $\partial^{A^T}_n\psi  = A^T\nabla\psi \cdot n$. With Cauchy--Schwarz and \eqref{eq:regular_parabolic_max}, this implies
\begin{align*} \|u\|_{L^2(I,L^2(\Omega))} 
    &\lesssim (\|u(0)\|_{L^2(\Omega)} + \|\partial_tu - \nabla\cdot (A\nabla u) \|_{W^*} 
    \\
    &+
    \| u \|_{{L^2(I, H^{1/2}(\partial\Omega))\cap H^{1/4}(I, L^2(\partial\Omega))^*}}).
\end{align*}
The estimate on the boundary term is due to the fact that $\partial^{A^T}_n\varphi$ is a member of the space $L^2(I, H^{1/2}(\partial\Omega))\cap H^{1/4}(I, L^2(\partial\Omega))$ for which we refer to \cite{weidemaier2002maximal}.\end{proof}

\subsection{Dual Estimate for Hyperbolic Equations}\label{sec:dual_hyperbolic}
\begin{lemma}
  For $u \in L^2(I, H^2(\Omega)) \cap H^2(I,L^2(\Omega))$, we have 
\begin{multline}
\|u\|_{H^1(I;L^2(\Omega))^*} \lesssim \|\partial_{tt} u - \nabla \cdot(A \nabla u )\|_{H^{2,2}(I \times \Omega)^*} + \|u\|_{H^{1/2,1/2}(\partial I \times \Omega)^*} \\  
+ \|u(0)\|_{H^{3/2}(\Omega)^*}+ \|\partial_t u(0)\|_{H^{1/2}(\Omega)^*}, 
\end{multline}
where $H^{2,2}(I \times \Omega) = L^2(I;H^2(\Omega)\cap H^1_0(\Omega)) \cap H^2(I;L^2(\Omega))$ and $H^{1/2,1/2}(I \times \partial \Omega) =L^2(0,T;H^{1/2}(\partial \Omega)) \cap H^{1/2}(0,T;L^2(\partial \Omega))$.
\end{lemma}
 \begin{proof}
We write 
\begin{align*}
 \|u\|_{H^1(I;L^2(\Omega))^*} = \sup_{\phi \in H^1(I;L^2(\Omega))} \frac{\int_I \int_{\Omega} u \phi }{\|\phi\|_{H^1(I;L^2(\Omega))}}
\end{align*}
For a given $\phi \in H^1(I;L^2(\Omega))$, define the backward in time problem 
\begin{align*}
\partial_{tt} \psi - \nabla \cdot (A^T \nabla  \psi)  = \phi ,  \quad \mathrm{in} \,\,  \Omega
\end{align*}
with zero boundary condition and zero final time conditions. Then, $\psi \in H^{2,2}(I \times \Omega) $ \cite[Chapter 5, Theorem 2.1]{lions2012non} and  
$$\|\psi\|_{H^{2,2}( I \times \Omega)} \lesssim \|\phi\|_{H^1(I; L^2(\Omega))}.$$
With Green's theorem, we can then write 
\begin{align*}
\int_I \int_{\Omega} u \phi  = \int_I \int_\Omega (\partial_{tt} u - \nabla \cdot(A \nabla u)  \psi + \int_I\int_{\partial\Omega}u\partial^{A^T}_n\psi \,\mathrm dx\mathrm dt  \\ + \int_{\Omega} (\psi(0) \partial_t u(0) - \partial_t \psi(0) u (0)) 
\end{align*}
From trace theory 
\cite[Chapter 4, Theorem 2.1]{lions2012non}, we have that 
\begin{align}
\partial^{A^T}_n\psi \in H^{1/2,1/2}( I \times \partial \Omega), \,\, \psi(0) \in H^{3/2}(\Omega), \,\, \partial_t \psi(0) \in H^{1/2}(\Omega). 
\end{align}
The result then readily follows, the details are skipped for brevity. 
\end{proof}

\subsection{Dual Estimate for Stokes Equations} \label{sec:regularity_stokes} 
\begin{lemma}
For all $ (\bm u,p) \in H^2(\Omega)^d \times (H^1(\Omega) \cap L^2_0(\Omega))$, there holds 
\begin{align} \label{eq:dual_bound_stokes}
        \begin{split}
            \|\bm u\|_{L^2(\Omega)} 
            +
            \| p \|_{(H^1(\Omega)\cap L^2_0(\Omega))^*}
            &\lesssim
            \|  -\Delta \bm{u} + \nabla p \|_{(H^2(\Omega)\cap H^1_0(\Omega))^*}
            \\
            &+ \| \nabla \cdot \bm{u}\|_{(H^1(\Omega)\cap L^2_0(\Omega))^*} + 
            \|\bm{u} \|_{H^{1/2}(\partial\Omega)^*}
           .
        \end{split}
    \end{align}
\end{lemma}

\begin{proof}  Such a result in 2D can be found in \cite{aziz1985least}. For completeness,  we present a proof here for any $d$. We write
\begin{align}
    \|\bm{u}\|_{L^2(\Omega)}  +\|p\|_{(H^{1}(\Omega) \cap L^2_0(\Omega))^*} = \sup_{\bm{\phi}_1 \in L^2(\Omega)} \frac{(\bm{u},\bm{\phi_1})_{\Omega}}{\|\bm{\phi_1}\|_{L^2(\Omega)}} + \sup_{\phi_2 \in H^1(\Omega) \cap L^2_0(\Omega) } \frac{(p,\phi_2)_\Omega}{\|\phi_2\|_{H^1(\Omega)}}. \label{eq:def_negative_norm_2} 
\end{align}
To simplify notation, we define  $\bm{f}_1 = -\Delta \bm{u} + \nabla p$, 
   $ f_2  = \nabla \cdot \bm{u}$, and  
   $ \bm{g} = \bm u \vert_{\partial \Omega}$. For any $(\bm{\phi}_1 , \phi_2) \in L^2(\Omega)^d \times (H^1(\Omega) \cap L^2_0(\Omega))$, let $(\bm{\psi}, q) \in H_0^1(\Omega)^d \times (H^1(\Omega) \cap L^2_0(\Omega))$ satisfy
\begin{align}
    -\Delta \bm{\psi} + \nabla q & = \bm{\phi}_1, \quad \mathrm{in} \,\,  \Omega, \\ 
    \nabla \cdot \bm{\psi} & = - \phi_2,  \quad \mathrm{in} \,\, \Omega,   \\ 
    \bm{\psi}  & = 0,  \quad \mathrm{on} \,\,  \partial \Omega . 
\end{align}
From the regularity results for Stokes \cite[Proposition 2.2]{temam2001navier}, we also have that
\begin{equation}
\|\bm{\psi}\|_{H^{2}(\Omega)} + \|q\|_{H^{1}(\Omega)} \lesssim \|\bm{\phi}_1\|_{L^2(\Omega)} + \|\phi_2\|_{H^1(\Omega)}. \label{eq:regularity_dual_stokes}
 \end{equation}
 We then evaluate 
 \begin{align*}
    (\bm{u},\bm{\phi}_1) + (p,\phi_2) & = ( \bm{u}, -\Delta \bm{\psi} + \nabla q ) -  (\nabla \cdot \bm{\psi}, p)  \\ & = (-\Delta \bm{u}, \bm{\psi}) - (\nabla \bm{\psi} \bm{n}, \bm{g})_{\partial \Omega} + (\nabla \cdot \bm{u}, q) - (\bm{g}\cdot \bm{n},q)_{\partial \Omega} +  (\bm \psi, \nabla p)\\ 
    & = (\bm{f}_1, \bm{\psi})- (\nabla \bm{\psi} \bm{n}, \bm{g})_{\partial \Omega}+ (f_2, q) - (\bm{g}\cdot \bm{n},q)_{\partial \Omega} . 
    \end{align*}
    Hence, we have the bound:
\begin{multline} \nonumber
|(\bm{u}, \bm{\phi_1})| +|(p,\phi_2)| \leq \|\bm{f}_1\|_{(H^{2}(\Omega)\cap H^1_0(\Omega))^* } \|\bm{\psi} \|_{H^{2}(\Omega)} + \| \bm g\|_{H^{1/2}(\partial \Omega)^*} \|\nabla \bm{\psi} \, \bm n \|_{H^{1/2}(\partial \Omega)}   \\ + \|f_2\|_{(H^{1} (\Omega) \cap L^2_0(\Omega))^*} \|q \|_{H^{1}(\Omega)} + \|\bm g\|_{H^{1/2}(\partial \Omega)^*}\|q\|_{H^{1/2}(\partial \Omega)}. 
\end{multline}
Along with the trace estimates 
\[ \|\nabla \bm \psi \,  \bm n \|_{H^{1/2}(\partial \Omega)} \lesssim \| \bm \psi\|_{H^{2}(\Omega)}, \quad \|q\|_{H^{1/2}(\partial \Omega)} \lesssim \|q\|_{H^1(\Omega)}, \]
and bounds \eqref{eq:regularity_dual_stokes} and \eqref{eq:def_negative_norm_2}, we choose $\phi_2 = 0$ to conclude that bound  \eqref{eq:dual_bound_stokes} holds for the first term holds,  and we choose $\bm{\phi}_1 = \bm{0}$ to obtain the bound on the second term. 
   \end{proof}

\section{Further Numerical Experiments}
In this Section we present further numerical experiments. In Table \ref{table:errors_accurate_1} and \ref{table:errors_accurate_2}, we report errors for the same equations as considered in Section \ref{sec:numerics} with manufactured solutions \ref{table:manufactured_soltuions}. In this case we run the optimization for 5000 iterations. We see that we can improve the observed errors significantly. Again, the crucial ingredient for successful optimization is the recently proposed natural gradient method for the optimization of PINNs \cite{muller2023achieving}.

\begin{table}[h!]
\begin{tabular}{|c|c|c|c|} \hline 
    PDE        &  $w$  & $L^2$ error & $H^1$ error       \\  \hline  \hline 
    Poisson    &  64   & $1.84e-06$  & $1.13e-05$        \\  \hline 
    Elasticity &  64   & $2.06e-04$  & $6.03e-03$        \\  \hline 
    Parabolic  &  64   & $6.17e-05$  & $5.34e-04$        \\  \hline 
    Hyperbolic &  64   & $3.12e-05$  & $4.42e-04$        \\  \hline
\end{tabular}    
\caption{Computed $L^2$ and $H^1$ errors of the PINN formulation using a $\tanh$ shallow neural network of width $w$ \emph{with a bigger computational budget}. The manufactured solutions considered are given in Table~\ref{table:manufactured_soltuions} }
\label{table:errors_accurate_1}
\end{table}

\begin{table}[h!]
\begin{tabular}{|c|c|c|c|c|} \hline 
    PDE                & $w_1$ & $w_2$    & $L^2$  & $H^1$    \\  \hline  \hline 
 \multirow{2}{*}{Stokes}  & \multirow{2}{*}{32}    &  \multirow{2}{*}{32}        &  \multirow{2}{*}{$\bm u:\,\, 1.41e-04$}   & $\bm u:\,\, 2.19e-03$    \\   
                  &     &        &    & $p: \,\, 1.93e-02$    \\  \hline 
    \multirow{2}{*}{Transient Stokes}  & \multirow{2}{*}{32}    &  \multirow{2}{*}{32}        &  \multirow{2}{*}{$\bm u:\,\,  9.69e-05$}   & $\bm u:\,\, 1.62e-03$    \\   
                 &     &        &    & $p: \,\, 1.37e-02$    \\  \hline 
 \multirow{2}{*}{Darcy}  & \multirow{2}{*}{32}    &  \multirow{2}{*}{32}        &  $\bm \sigma:\,\,  1.01e-04$  & \multirow{2}{*}{--}\\   
                 &     &        &  $p: \,\, 7.79-06$  &     \\  \hline    
\multirow{2}{*}{Inverse}  & \multirow{2}{*}{32}    &  \multirow{2}{*}{32}        &  $\bm u:\,\, 7.88e-05 $  & \multirow{2}{*}{$\bm u: \,\, 1.36e-03$ }  \\   
                  &     &        &  $f: \,\, 1.20e-03$  &     \\  \hline    \end{tabular}    
\caption{Computed $L^2$ and $H^1$ errors of the PINN formulation using a $\tanh$ shallow neural network of width $w$ \emph{with a bigger computational budget}. The manufactured solutions considered are given in Table~\ref{table:manufactured_soltuions}. The  error in the pressure for the steady and transient Stokes equations is measured in the $H^1(\Omega)$ and $L^2(I,H^1(\Omega))$ semi-norms, respectively.}
\label{table:errors_accurate_2}
\end{table}

\bibliography{references}

\begin{thebibliography}{10}

\bibitem{amrouche1991existence}
Ch{\'e}rif Amrouche and Vivette Girault.
\newblock On the existence and regularity of the solution of {Stokes} problem in arbitrary dimension.
\newblock {\em Proc. Japan Acad. Ser. A Math. Sci}, 67 (5):171--175, 1991.

\bibitem{aziz1985least}
AK~Aziz, RB~Kellogg, and AB~Stephens.
\newblock Least squares methods for elliptic systems.
\newblock {\em Mathematics of Computation}, 44(169):53--70, 1985.

\bibitem{jax2018github}
James Bradbury, Roy Frostig, Peter Hawkins, Matthew~James Johnson, Chris Leary, Dougal Maclaurin, George Necula, Adam Paszke, Jake Vander{P}las, Skye Wanderman-{M}ilne, and Qiao Zhang.
\newblock {JAX}: composable transformations of {P}ython+{N}um{P}y programs, 2018.

\bibitem{bramble1970rayleigh}
James~H Bramble and Alfred~H Schatz.
\newblock Rayleigh-{R}itz-{G}alerkin methods for {D}irichlet's problem using subspaces without boundary conditions.
\newblock {\em Communications on Pure and Applied Mathematics}, 23(4):653--675, 1970.

\bibitem{brenner2008mathematical}
Susanne Brenner and Ridgway Scott.
\newblock {\em The Mathematical Theory of Finite Element Methods}.
\newblock Springer, 2008.

\bibitem{cai2021artificial}
Shengze Cai, He~Li, Fuyin Zheng, Fang Kong, Ming Dao, George~Em Karniadakis, and Subra Suresh.
\newblock Artificial intelligence velocimetry and microaneurysm-on-a-chip for three-dimensional analysis of blood flow in physiology and disease.
\newblock {\em Proceedings of the National Academy of Sciences}, 118(13):e2100697118, 2021.

\bibitem{cai2021physics}
Shengze Cai, Zhiping Mao, Zhicheng Wang, Minglang Yin, and George~Em Karniadakis.
\newblock Physics-informed neural networks {(PINNs)} for fluid mechanics: A review.
\newblock {\em Acta Mechanica Sinica}, 37(12):1727--1738, 2021.

\bibitem{cobos1998theorem}
Fernando Cobos and Tomas Schonbek.
\newblock On a theorem by {L}ions and {P}eetre about interpolation between a {B}anach space and its dual.
\newblock {\em Houston J. Math}, 24(2):325--344, 1998.

\bibitem{10.1093/imanum/drac085}
Tim De~Ryck, Ameya~D Jagtap, and Siddhartha Mishra.
\newblock {Error estimates for physics-informed neural networks approximating the {Navier–Stokes} equations}.
\newblock {\em IMA Journal of Numerical Analysis}, page drac085, 01 2023.

\bibitem{de2022error}
Tim De~Ryck and Siddhartha Mishra.
\newblock Error analysis for physics-informed neural networks ({PINN}s) approximating {Kolmogorov PDEs}.
\newblock {\em Advances in Computational Mathematics}, 48(6):79, 2022.

\bibitem{dissanayake1994neural}
MWMG Dissanayake and Nhan Phan-Thien.
\newblock Neural-network-based approximations for solving partial differential equations.
\newblock {\em Communications in Numerical Methods in Engineering}, 10(3):195--201, 1994.

\bibitem{ern2004theory}
Alexandre Ern and Jean-Luc Guermond.
\newblock {\em Theory and Practice of Finite Elements}, volume 159.
\newblock Springer, 2004.

\bibitem{evans2022partial}
Lawrence~C Evans.
\newblock {\em Partial Differential Equations}, volume~19.
\newblock American Mathematical Society, 2022.

\bibitem{guermond2009lbb}
J-L Guermond.
\newblock The {LBB} condition in fractional {S}obolev spaces and applications.
\newblock {\em IMA Journal of Numerical Analysis}, 29(3):790--805, 2009.

\bibitem{GR20}
Ingo G{\"u}hring and Mones Raslan.
\newblock Approximation rates for neural networks with encodable weights in smoothness spaces.
\newblock {\em Neural Networks}, 134:107--130, 2021.

\bibitem{haghighat2021physics}
Ehsan Haghighat, Maziar Raissi, Adrian Moure, Hector Gomez, and Ruben Juanes.
\newblock A physics-informed deep learning framework for inversion and surrogate modeling in solid mechanics.
\newblock {\em Computer Methods in Applied Mechanics and Engineering}, 379:113741, 2021.

\bibitem{han2018solving}
Jiequn Han, Arnulf Jentzen, and Weinan E.
\newblock Solving high-dimensional partial differential equations using deep learning.
\newblock {\em Proceedings of the National Academy of Sciences}, 115(34):8505--8510, 2018.

\bibitem{han2017deep}
Jiequn Han, Arnulf Jentzen, et~al.
\newblock Deep learning-based numerical methods for high-dimensional parabolic partial differential equations and backward stochastic differential equations.
\newblock {\em Communications in Mathematics and Statistics}, 5(4):349--380, 2017.

\bibitem{hennigh2021nvidia}
Oliver Hennigh, Susheela Narasimhan, Mohammad~Amin Nabian, Akshay Subramaniam, Kaustubh Tangsali, Zhiwei Fang, Max Rietmann, Wonmin Byeon, and Sanjay Choudhry.
\newblock Nvidia simnet™: An ai-accelerated multi-physics simulation framework.
\newblock In {\em International conference on computational science}, pages 447--461. Springer, 2021.

\bibitem{hermann2020deep}
Jan Hermann, Zeno Sch{\"a}tzle, and Frank No{\'e}.
\newblock Deep-neural-network solution of the electronic {S}chr{\"o}dinger equation.
\newblock {\em Nature Chemistry}, 12(10):891--897, 2020.

\bibitem{hu2023solving}
Tianhao Hu, Bangti Jin, and Zhi Zhou.
\newblock Solving {P}oisson problems in polygonal domains with singularity enriched physics informed neural networks.
\newblock {\em arXiv preprint arXiv:2308.16429}, 2023.

\bibitem{hytonen2016analysis}
Tuomas Hyt{\"o}nen, Jan Van~Neerven, Mark Veraar, and Lutz Weis.
\newblock {\em Analysis in Banach spaces}, volume~12.
\newblock Springer, 2016.

\bibitem{jagtap2020adaptive}
Ameya~D Jagtap, Kenji Kawaguchi, and George~Em Karniadakis.
\newblock Adaptive activation functions accelerate convergence in deep and physics-informed neural networks.
\newblock {\em Journal of Computational Physics}, 404:109136, 2020.

\bibitem{jin2021nsfnets}
Xiaowei Jin, Shengze Cai, Hui Li, and George~Em Karniadakis.
\newblock {NSF}nets ({Navier-Stokes} flow nets): Physics-informed neural networks for the incompressible {Navier-Stokes} equations.
\newblock {\em Journal of Computational Physics}, 426:109951, 2021.

\bibitem{karniadakis2021physics}
George~Em Karniadakis, Ioannis~G Kevrekidis, Lu~Lu, Paris Perdikaris, Sifan Wang, and Liu Yang.
\newblock Physics-informed machine learning.
\newblock {\em Nature Reviews Physics}, 3(6):422--440, 2021.

\bibitem{krishnapriyan2021characterizing}
Aditi Krishnapriyan, Amir Gholami, Shandian Zhe, Robert Kirby, and Michael~W Mahoney.
\newblock Characterizing possible failure modes in physics-informed neural networks.
\newblock {\em Advances in Neural Information Processing Systems}, 34:26548--26560, 2021.

\bibitem{lagaris1998artificial}
Isaac~E Lagaris, Aristidis Likas, and Dimitrios~I Fotiadis.
\newblock Artificial neural networks for solving ordinary and partial differential equations.
\newblock {\em IEEE transactions on neural networks}, 9(5):987--1000, 1998.

\bibitem{lions2012non}
Jacques~Louis Lions and Enrico Magenes.
\newblock {\em Non-homogeneous Boundary Value Problems and Applications: Vol. 2}.
\newblock Springer Science \& Business Media, 1972.

\bibitem{mclean2000strongly}
William Charles~Hector McLean.
\newblock {\em Strongly Elliptic Systems and Boundary Integral Equations}.
\newblock Cambridge university press, 2000.

\bibitem{mishraforward}
Siddhartha Mishra and Roberto Molinaro.
\newblock {Estimates on the generalization error of physics-informed neural networks for approximating {PDE}s}.
\newblock {\em IMA Journal of Numerical Analysis}, 43(1):1--43, 01 2022.

\bibitem{muller2022error}
Johannes M{\"u}ller and Marius Zeinhofer.
\newblock Error estimates for the deep {R}itz method with boundary penalty.
\newblock In {\em Mathematical and Scientific Machine Learning}, pages 215--230. PMLR, 2022.

\bibitem{muller2022notes}
Johannes M{\"u}ller and Marius Zeinhofer.
\newblock Notes on exact boundary values in residual minimisation.
\newblock In {\em Mathematical and Scientific Machine Learning}, pages 231--240. PMLR, 2022.

\bibitem{muller2023achieving}
Johannes M{\"u}ller and Marius Zeinhofer.
\newblock Achieving high accuracy with {PINN}s via energy natural gradient descent.
\newblock {\em ICML}, 2023.

\bibitem{nochetto2015pde}
Ricardo~H Nochetto, Enrique Ot{\'a}rola, and Abner~J Salgado.
\newblock A {PDE} approach to fractional diffusion in general domains: a priori error analysis.
\newblock {\em Foundations of Computational Mathematics}, 15:733--791, 2015.

\bibitem{pehlivanov1994least}
AI~Pehlivanov, GF~Carey, and RD~Lazarov.
\newblock Least-squares mixed finite elements for second-order elliptic problems.
\newblock {\em SIAM Journal on Numerical Analysis}, 31(5):1368--1377, 1994.

\bibitem{pfau2020ab}
David Pfau, James~S Spencer, Alexander~GDG Matthews, and W~Matthew~C Foulkes.
\newblock Ab initio solution of the many-electron {S}chr{\"o}dinger equation with deep neural networks.
\newblock {\em Physical Review Research}, 2(3):033429, 2020.

\bibitem{raissi2019physics}
Maziar Raissi, Paris Perdikaris, and George~E Karniadakis.
\newblock Physics-informed neural networks: A deep learning framework for solving forward and inverse problems involving nonlinear partial differential equations.
\newblock {\em Journal of Computational physics}, 378:686--707, 2019.

\bibitem{richter2022neural}
Jack Richter-Powell, Yaron Lipman, and Ricky~TQ Chen.
\newblock Neural conservation laws: A divergence-free perspective.
\newblock {\em Advances in Neural Information Processing Systems}, 35:38075--38088, 2022.

\bibitem{schechter1963p}
Martin Schechter.
\newblock On ${L}^p$ estimates and regularity {II}.
\newblock {\em Mathematica Scandinavica}, 13(1):47--69, 1963.

\bibitem{shin2020convergence}
Yeonjong Shin, Jerome Darbon, and George~Em Karniadakis.
\newblock On the convergence of physics informed neural networks for linear second-order elliptic and parabolic type {PDE}s.
\newblock {\em arXiv preprint arXiv:2004.01806}, 2020.

\bibitem{shin2020error}
Yeonjong Shin, Zhongqiang Zhang, and George~Em Karniadakis.
\newblock Error estimates of residual minimization using neural networks for linear {PDEs}.
\newblock {\em arXiv preprint arXiv:2010.08019}, 2020.

\bibitem{siegel2023greedy}
Jonathan~W Siegel, Qingguo Hong, Xianlin Jin, Wenrui Hao, and Jinchao Xu.
\newblock Greedy training algorithms for neural networks and applications to {PDEs}.
\newblock {\em Journal of Computational Physics}, 484:112084, 2023.

\bibitem{siegel2020approximation}
Jonathan~W Siegel and Jinchao Xu.
\newblock Approximation rates for neural networks with general activation functions.
\newblock {\em Neural Networks}, 128:313--321, 2020.

\bibitem{siegel2022high}
Jonathan~W Siegel and Jinchao Xu.
\newblock High-order approximation rates for shallow neural networks with cosine and {ReLUk} activation functions.
\newblock {\em Applied and Computational Harmonic Analysis}, 58:1--26, 2022.

\bibitem{sukumar2022exact}
N~Sukumar and Ankit Srivastava.
\newblock Exact imposition of boundary conditions with distance functions in physics-informed deep neural networks.
\newblock {\em Computer Methods in Applied Mechanics and Engineering}, 389:114333, 2022.

\bibitem{temam2001navier}
Roger Temam.
\newblock {\em Navier-Stokes equations: Theory and Numerical Analysis}, volume 343.
\newblock American Mathematical Soc., 2001.

\bibitem{wang2021understanding}
Sifan Wang, Yujun Teng, and Paris Perdikaris.
\newblock Understanding and mitigating gradient flow pathologies in physics-informed neural networks.
\newblock {\em SIAM Journal on Scientific Computing}, 43(5):A3055--A3081, 2021.

\bibitem{wang2022and}
Sifan Wang, Xinling Yu, and Paris Perdikaris.
\newblock When and why {PINN}s fail to train: {A} neural tangent kernel perspective.
\newblock {\em Journal of Computational Physics}, 449:110768, 2022.

\bibitem{weidemaier2002maximal}
Peter Weidemaier.
\newblock Maximal regularity for parabolic equations with inhomogeneous boundary conditions in {S}obolev spaces with mixed {$L_p$}-norm.
\newblock {\em Electronic Research Announcements of the American Mathematical Society}, 8(6):47--51, 2002.

\bibitem{wu2023comprehensive}
Chenxi Wu, Min Zhu, Qinyang Tan, Yadhu Kartha, and Lu~Lu.
\newblock A comprehensive study of non-adaptive and residual-based adaptive sampling for physics-informed neural networks.
\newblock {\em Computer Methods in Applied Mechanics and Engineering}, 403:115671, 2023.

\bibitem{xu2020finite}
Jinchao Xu.
\newblock The finite neuron method and convergence analysis.
\newblock {\em arXiv preprint arXiv:2010.01458}, 2020.

\bibitem{yu2018deep}
Bing Yu et~al.
\newblock The deep {R}itz method: a deep learning-based numerical algorithm for solving variational problems.
\newblock {\em Communications in {M}athematics and {S}tatistics}, 6(1):1--12, 2018.

\end{thebibliography}
\bibliographystyle{plain}

\end{document}